\documentclass[10pt,reqno]{amsart}
\usepackage{hyperref}
\usepackage{latexsym,amsmath}
\usepackage{enumerate}
\usepackage{amsfonts}
\usepackage{amssymb}
\usepackage{latexsym}
\usepackage{fixmath}
 \usepackage[usenames,dvipsnames]{color}

\usepackage{graphicx}

\usepackage[T1]{fontenc}
\usepackage{fourier}
\usepackage{bbm}
\usepackage{color}
\usepackage{fullpage}

\usepackage{pgf,pgfplots}
\usepackage{tikz,units}
\usetikzlibrary{arrows}

\numberwithin{equation}{section}

\newcommand\RSA{Random Structures and Algorithms}

\newcommand\CPC{Combinatorics, Probability and Computing}

\renewcommand{\vec}[1]{\boldsymbol{#1}}

\renewcommand{\subset}{\subseteq}

\newcommand\disteq{\,\stacksign{d}=\,}

\newcommand{\BP}{\mathrm{BP}}
\newcommand{\DE}{\BP}  
\newcommand{\LDE}{\mathrm{LL}}
\newcommand{\LDEP}{\mathrm{LL}^+}
\newcommand{\LDEPl}{\mathrm{LL}^{+\,(\ell)}}
\newcommand{\LDEPk}{\mathrm{LL}^{+\,(k)}}

\newcommand\vX{\vec X}

\newcommand\vd{\vec d}
\newcommand\vy{\vec y}
\newcommand\vz{\vec z}

\newcommand\vI{\vec I}

\newcommand\MU{\vec\mu}
\newcommand\ETA{\vec\eta}
\newcommand\vY{\vec Y}
\newcommand\vm{{\vec m}}
\newcommand\vs{{\vec s}}
\newcommand\vT{{\vec T}}

\newcommand\CHI{{\vec\chi}}
\newcommand\DELTA{{\vec\Delta}}

\newcommand\PHI{\vec\Phi}
\newcommand\nix{\,\cdot\,}

\newcommand\vA{\vec A}

\newcommand\dd{{\mathrm d}}

\newcommand\SIGMA{\vec\sigma}
\newcommand\TAU{\vec\tau}

\newcommand\fB{\mathfrak{B}}

\newcommand\cC{\mathcal{C}}

\newcommand\cF{\mathcal{F}}
\newcommand\cG{\mathcal{G}}
\newcommand\cE{\mathcal{E}}
\newcommand\cU{\mathcal{U}}

\newcommand\cI{\mathcal{I}}

\newcommand\cP{\mathcal{P}}

\newcommand\cW{\mathcal{W}}

\def\cC{{\mathcal C}}
\def\cE{{\mathcal E}}

\newcommand\eps{\varepsilon}

\newcommand\Erw{\mathbb{E}}
\newcommand{\vecone}{\vec{1}}

\newcommand{\Po}{{\rm Po}}
\newcommand{\Bin}{{\rm Bin}}

\newcommand\dTV{d_{\mathrm{TV}}}

\newcommand\bc[1]{\left({#1}\right)}
\newcommand\cbc[1]{\left\{{#1}\right\}}
\newcommand\bcfr[2]{\bc{\frac{#1}{#2}}}

\newcommand\brk[1]{\left\lbrack{#1}\right\rbrack}

\newcommand\abs[1]{\left|{#1}\right|}

\newcommand\RR{\mathbb{R}}

\newcommand{\whp}{w.h.p.}

\newcommand{\stacksign}[2]{{\stackrel{\mbox{\scriptsize #1}}{#2}}}

\newcommand{\Erdos}{Erd\H os}
\newcommand{\Renyi}{R\'enyi}

\newcommand{\Bollobas}{Bollob\'as}
\newcommand{\Chvatal}{Chv\'{a}tal}

\newcommand{\Luczak}{\L uczak}

\newcommand\pr{\mathbb{P}} 
\renewcommand\Pr{\pr} 
\newcommand\Lem{Lemma}
\newcommand\Prop{Proposition}
\newcommand\Thm{Theorem}

\newcommand\Cor{Corollary}
\newcommand\Sec{Section}
\newcommand\Chap{Chapter}

\newtheorem{definition}{Definition}[section]
\newtheorem{claim}[definition]{Claim}

\newtheorem{theorem}[definition]{Theorem}
\newtheorem{lemma}[definition]{Lemma}
\newtheorem{proposition}[definition]{Proposition}
\newtheorem{corollary}[definition]{Corollary}

\newtheorem{fact}[definition]{Fact}

\newtheoremstyle{case}{}{}{}{}{}{:}{ }{}
\theoremstyle{case}

\DeclareMathOperator{\sign}{sign}

\newcommand{\vn}{{\vec n}}

\newcommand{\vN}{\vec N}

\begin{document}

\title{The random $2$-SAT partition function}

\author{Dimitris~Achlioptas, Amin~Coja-Oghlan, 
Max~Hahn-Klimroth, Joon~Lee, No\"ela~M\"uller, Manuel~Penschuck, Guangyan~Zhou}

\address{Dimitris Achlioptas, {\tt optas@di.uoa.gr}, University of Athens, Department of Computer Science \& Telecommunications, Pan\-e\-pisti\-mi\-opolis, Ilissia,
Athens 15784, Greece.}

\address{Amin Coja-Oghlan, {\tt acoghlan@math.uni-frankfurt.de}, Goethe University, Mathematics Institute, 10 Robert Mayer St, Frankfurt 60325, Germany.}



\address{Max Hahn-Klimroth, {\tt hahnklim@math.uni-frankfurt.de}, Goethe University, Mathematics Institute, 10 Robert Mayer St, Frankfurt 60325, Germany.}

\address{Joon Lee, {\tt lee@math.uni-frankfurt.de}, Goethe University, Mathematics Institute, 10 Robert Mayer St, Frankfurt 60325, Germany.}

\address{No\"ela M\"uller, {\tt nmueller@math.uni-frankfurt.de}, Goethe University, Mathematics Institute, 10 Robert Mayer St, Frankfurt 60325, Germany.}

\address{Manuel Penschuck, {\tt manuel@ae.cs.uni-frankfurt.de}, Goethe University, Computer Science Institute, 11--15 Robert Mayer St, Frankfurt 60325, Germany.}

\address{Guangyan Zhou, {\tt gyzhou76@gmail.com}, 
School of Mathematics and Statistics, Beijing Technology and Business University, Beijing 100048, China.}

\thanks{Amin Coja-Oghlan's research received support under DFG CO 646/4.
Max Hahn-Klimroth has been supported by Stiftung Polytechnische Gesellschaft.
Manuel Penschuck's research received support under DFG ME 2088/3-2 and ME 2088/4-2.
Guangyan Zhou is supported by National Natural Science Foundation of China, No.\ 61702019.
}

\begin{abstract}
We show that throughout the satisfiable phase the normalised number of satisfying assignments of a random $2$-SAT formula converges in probability to an expression predicted by the cavity method from statistical physics.
The proof is based on showing that the Belief Propagation algorithm renders the correct marginal probability that a variable is set to `true' under a uniformly random satisfying assignment.
\hfill
\end{abstract}

\maketitle

\section{Introduction}\label{Sec_intro}

\subsection{Background and motivation}
\noindent
The random $2$-SAT problem was the first random constraint satisfaction problem whose satisfiability threshold could be pinpointed precisely, an accomplishment attained independently by \Chvatal\ and Reed~\cite{CR} and Goerdt~\cite{Goerdt} in 1992.
The proofs evince the link between the $2$-SAT threshold and the percolation phase transition of a random digraph.
This connection subsequently enabled  \Bollobas,  Borgs, Chayes, Kim and Wilson~\cite{BBCKW} to identify the size of the scaling window, which matches that of the giant component phase transition of the \Erdos-\Renyi{} random graph~\cite{B,Luczak}.
Ramifications and extensions	 of these results pertain to random $2$-SAT formulas with given literal degrees~\cite{CFS},
the random MAX $2$-SAT problem~\cite{CGHS} and the performance of algorithms~\cite{SS}.
But despite the great attention devoted to random $2$-SAT over the years, a fundamental question, mentioned prominently in the survey~\cite{Lalo}, remained conspicuously open:
{\em how many satisfying assignments does a random $2$-SAT formula typically possess?}
While percolation-type arguments have been stretched to derive (rough) bounds~\cite{BD}, the exact answer remained beyond the reach of elementary techniques.

In addition to the mathematical literature, the $2$-SAT problem attracted the interest of statistical physicists, who brought to bear a canny but non-rigorous approach called the cavity method~\cite{MZ,MZ2}.
Instead of relying on percolation ideas, the physics {\em ansatz} seizes upon a heuristic message passing scheme called Belief Propagation.
Its purpose is to calculate the marginal probabilities that a random satisfying assignment sets specific variables of the $2$-SAT formula to `true'.
According to physics intuition Belief Propagation reveals a far more fine-grained picture than  a mere percolation argument possibly could.
Indeed, in combination with a functional called the Bethe free entropy, Belief Propagation renders a precise conjecture as to the number of satisfying assignments.

We prove this conjecture.
Specifically, we show that for all clause-to-variable densities below the $2$-SAT threshold the number of satisfying assignments is determined by the Bethe functional applied to a particular solution of a stochastic fixed point equation that mimics Belief Propagation.
The formula that we obtain does not boil down to a simple algebraic expression, which may explain why the problem has confounded classical methods for nearly three decades.
Nonetheless, thanks to rapid convergence of the stochastic fixed point iteration, the formula can be evaluated numerically within arbitrary precision.
A crucial step towards the main theorem is to verify that Belief Propagation does indeed yield the correct marginals,  a fact that may be of independent interest.

By comparison to prior work on Belief Propagation in combinatorics (e.g.,~\cite{CKPZ,dembo,DM,MS}), we face the substantial technical challenge of dealing with the `hard' constraints of the $2$-SAT problems, which demands that {\em all} clauses be satisfied.
A second novelty is that in order to prove convergence of Belief Propagation to the correct marginals we need to investigate delicately constructed extremal boundary conditions.
Since these depend on the random $2$-SAT formula itself, we need to develop means to confront the ensuing stochastic dependencies between the construction of the boundary condition and the subsequent message passing iterations.
We proceed to state the main results precisely.
An outline of the proofs and a detailed discussion of related work follow in \Sec s~\ref{Sec_outline} and~\ref{Sec_related}.

\subsection{The main result}
Let $n>1$ be an integer, let $d>0$ be a positive real and let $\vm\disteq\Po(dn/2)$ be a Poisson random variable.
Further, let $\PHI=\PHI_n$ be a random $2$-SAT formula with Boolean variables $x_1,\ldots,x_n$ and $\vm$ clauses,
drawn uniformly and independently from the set of all $4n(n-1)$ possible clauses with two distinct variables.
Thus, each variable appears in $d$ clauses on the average and the satisfiability threshold occurs at $d=2$.
We aim to estimate the number $Z(\PHI)$ of satisfying assignments, the {\em partition function} in physics jargon.
More precisely, since $Z(\PHI)$ remains exponentially large  for all $d<2$ \whp, in order to obtain a well-behaved limit we compute the normalised logarithm $n^{-1}\log Z(\PHI)$.

The result comes in terms of the solution to a stochastic fixed point equation on the unit interval.
Hence, let $\cP(0,1)$ be the set of all Borel probability measures on $(0,1)$, endowed with the weak topology.
Further, define an operator $\DE_d:\cP(0,1)\to\cP(0,1)$, $\pi\mapsto\hat\pi$ as follows.
With $\vd^+,\vd^-$ Poisson variables with mean $d/2$ and $\MU_{\pi,1},\MU_{\pi,2},\ldots$ random variables with distribution $\pi$,  all mutually independent, let $\hat\pi$ be the distribution of the random variable
\begin{align}\label{eqdensityEv}
\frac{\prod_{i=1}^{\vd^-}\MU_{\pi,i}}
{\prod_{i=1}^{\vd^-}\MU_{\pi,i}+\prod_{i=1}^{\vd^+}\MU_{\pi,i+\vd^-}}\in(0,1).
\end{align}
Let $\delta_{1/2}\in\cP(0,1)$ signify the atom at $1/2$ and write $\DE_d^\ell(\nix)$ for the $\ell$-fold application of the operator $\DE_d$.

\begin{theorem}\label{Thm_Z}
For any $d<2$ the limit $\pi_d=\lim_{\ell\to\infty}\BP_d^\ell(\delta_{1/2})$ exists and
\begin{align}\label{eqThm_Z}
\lim_{n\to\infty}\frac{1}{n}\log Z(\PHI)=
\Erw\brk{\log\bc{\prod_{i=1}^{\vd^-}\MU_{\pi_d,i}+\prod_{i=1}^{\vd^+}\MU_{\pi_d,i+\vd^-}}
-\frac d2\log\bc{1-\MU_{\pi_d,1}\MU_{\pi_d,2}}}
&&\mbox{in probability}.
\end{align}
\end{theorem}

\noindent 
Of course, the fact that the r.h.s.\ of \eqref{eqThm_Z} is well-defined is part of the statement of \Thm~\ref{Thm_Z}.

\begin{figure}
\includegraphics[height=50mm]{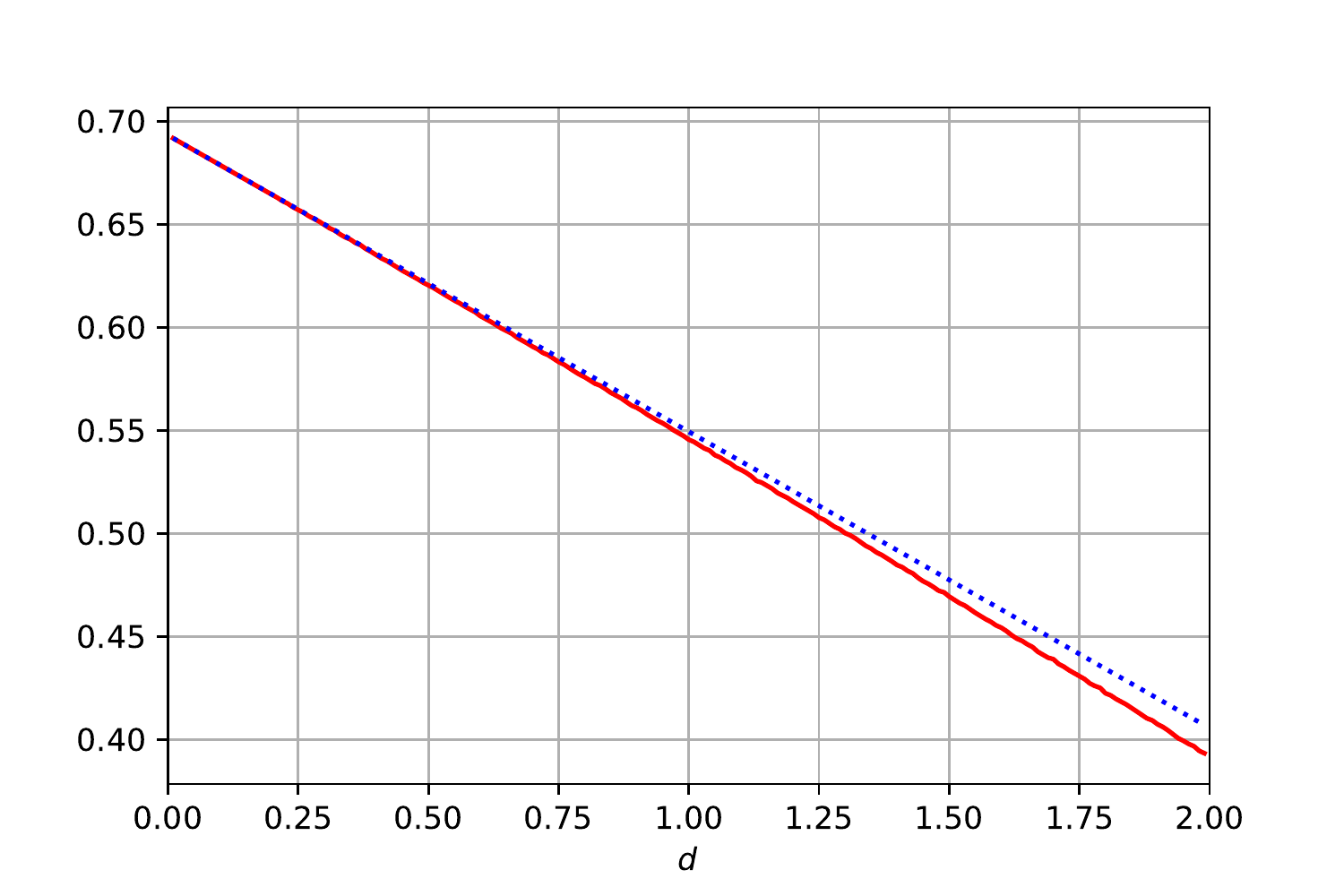}
\hfill\includegraphics[height=50mm]{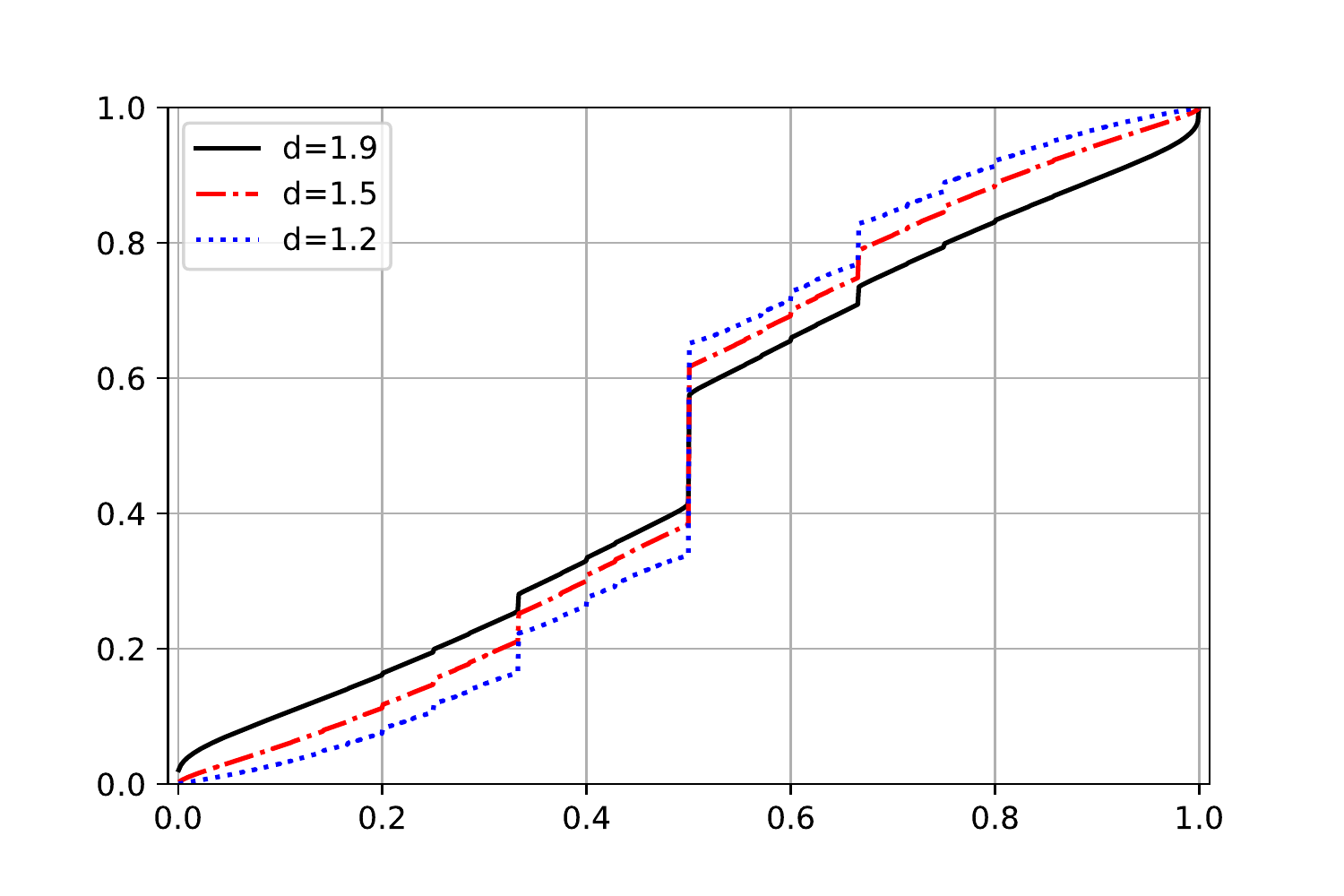}
\caption{{\em Left:} the red line depicts a numerical approximation to the r.h.s.\ of~\eqref{eqThm_Z} after 24 iterations of $\DE_d(\nix)$.
The dotted blue line displays the first moment bound.
{\em Right:} the cumulative density functions of numerical approximations to $\BP_d^{24}(\delta_{1/2})$ for various $d$.}\label{Fig_Z}
\end{figure}

By construction, the distribution $\pi_d$ is a solution to the stochastic fixed point equation
\begin{align}\label{eqDE}
\pi_d&=\BP_d(\pi_d).
\end{align}
The equation \eqref{eqDE} is known as the {\em density evolution} equation in physics lore, while the expression on the r.h.s.\ of \eqref{eqThm_Z} is called the Bethe free entropy~\cite{MM}.
Hence, \Thm~\ref{Thm_Z} matches the conjecture from~\cite{MZ}.
By comparison, Markov's inequality yields the elementary first moment bound
\begin{align}\label{eqannealed}
\frac1n\log Z(\PHI)\leq \frac1n\log \Erw[Z(\PHI)]+o(1)=(1-d)\log 2+\frac d2\log3+o(1)&&\mbox\whp,
\end{align}
which, however, fails to be tight for any $0<d<2$~\cite{PanchenkoTalagrand}.
Furthermore, while~\eqref{eqThm_Z} may appear difficult to evaluate, the proof reveals that the fixed point iteration $\DE^\ell_d(\delta_{1/2})$ converges geometrically (in an appropriate metric).
In effect, decent numerical approximations can be obtained; see Figure~\ref{Fig_Z}.

For $d<1$ the random digraph on $\{x_1,\neg x_1,\ldots,x_n,\neg x_n\}$ obtained by inserting for each clause $l_1\vee l_2$ of $\PHI$ the two directed edges $\neg l_1\to l_2$, $\neg l_2\to l_1$ is sub-critical and the distribution $\pi_d$ is supported on a countable set.
In effect, for $d<1$ the formula \eqref{eqThm_Z} can be obtained via elementary counting arguments.
By contrast, the emergence of a weak giant component for $1<d<2$ turns the computation of $Z(\PHI)$ into a challenge.
Finally, for $d>2$ the digraph contains a strongly connected giant component \whp{}
Its long directed cycles likely cause contradictions, which is why satisfying assignments cease to exist.

An asymptotically tight upper bound on $n^{-1}\log Z(\PHI)$ could be obtained via the interpolation method from mathematical physics~\cite{FranzLeone,PanchenkoTalagrand}.
We will revisit this point in \Sec~\ref{Sec_related}.
Thus, the principal contribution of \Thm~\ref{Thm_Z} is the lower bound on $\log Z(\PHI)$.
The best prior lower bound was obtained by Boufkhad and Dubois~\cite{BD} in 1999 via percolation arguments.
However, this bound drastically undershoots the actual value from \Thm~\ref{Thm_Z}.
For instance, for $d=1.2$, \cite{BD} gives $n^{-1}\log Z(\PHI)\geq 0.072\ldots\enspace$, while actually $n^{-1}\log Z(\PHI)= 0.515\ldots$ \whp

\subsection{Belief Propagation}\label{Sec_BP}
To elaborate on the combinatorial meaning of the distribution $\pi_d$, we need to look into the Belief Propagation heuristic.
Instantiated to $2$-SAT, Belief Propagation is a message passing algorithm designed to approximate the marginal probability that a specific variable takes the value `true' under a random satisfying assignment.
While finding satisfying assignments of a given $2$-SAT formula is an easy computational task, calculating these marginals is not.
In fact, the problem is $\#$P-hard~\cite{Valiant}.
Nonetheless, we are going to prove that Belief Propagation approximates the marginals well on random formulas \whp{}

To introduce Belief Propagation, we associate a bipartite graph $G(\PHI)$ with the formula $\PHI$.
One vertex class $V_n=\{x_1,\ldots,x_n\}$ represents the propositional variables, the other class $F_{\vm}=\{a_1,\ldots,a_{\vm}\}$ represents the clauses.
Each clause $a_i$ is adjacent to the two variables that it contains.
We write $\partial v=\partial(\PHI,v)$ for the set of neighbours of a vertex $v$ of $G(\PHI)$.
Moreover, for $\ell\geq1$ let $\partial^\ell v$ signify the set of all vertices at distance precisely $\ell$ from $v$.

Associated with the edges of $G(\PHI)$, the Belief Propagation messages are probability distributions on the Boolean values `true' and `false'.
To be precise, any adjacent clause/variable pair $a,x$ comes with two messages, one directed from $a$ to $x$ and a reverse one from $x$ to $a$.
Encoding `true' and `false' by $\pm1$, we initialise all messages by
\begin{align}\label{eqBPinit}
\nu_{\PHI,a\to x}^{(0)}(\pm1)=\nu_{\PHI,x\to a}^{(0)}(\pm1)=1/2.
\end{align}
For $\ell\geq1$ the messages $\nu_{\PHI,a\to x}^{(\ell)},\nu_{\PHI,x\to a}^{(\ell)}$ are defined inductively.
Specifically, suppose that clause $a$ contains the two variables $x,y$.
Let $r,s\in\{\pm1\}$ indicate whether $x,y$ appear as positive or negative literals in $a$.
Then for $t=\pm1$ let
\begin{align}\label{eqBPrec}
\nu_{\PHI,a\to x}^{(\ell)}(t)&=
	\frac{1-\vecone\cbc{r\neq t}\nu_{\PHI,y\to a}^{(\ell-1)}(-s)}
				{1+\nu_{\PHI,y\to a}^{(\ell-1)}(s)},&
\nu_{\PHI,x\to a}^{(\ell)}(t)&=\frac{\prod_{b\in\partial x\setminus\cbc a}\nu_{\PHI,b\to x}^{(\ell)}(t)}
 	{\prod_{b\in\partial x\setminus\cbc a}\nu_{\PHI,b\to x}^{(\ell)}(1)+\prod_{b\in\partial x\setminus\cbc a}\nu_{\PHI,b\to x}^{(\ell)}(-1)}.
\end{align}
The last expression is deemed to equal $1/2$ if the denominator vanishes (which does not happen if $\PHI$ is satisfiable).
Finally, the Belief Propagation estimate of the marginal of a variable $x$ after $\ell$ iterations reads
\begin{align}\label{eqBPmarg}
\nu_{\PHI,x}^{(\ell)}(t)&=\frac{\prod_{a\in\partial x}\nu_{\PHI,a\to x}^{(\ell)}(t)}
 	{\prod_{a\in\partial x}\nu_{\PHI,a\to x}^{(\ell)}(1)+\prod_{a\in\partial x}\nu_{\PHI,a\to x}^{(\ell)}(-1)},
\end{align}
again interpreted to yield $1/2$ if the denominator vanishes.
For an excellent exposition of Belief Propagation, including the derivation of \eqref{eqBPrec}--\eqref{eqBPmarg}, we point to~\cite[\Chap~14]{MM}.

The next theorem establishes that \eqref{eqBPmarg} approximates the true marginals well for large $\ell$.
In fact, we prove a significantly stronger result.
To set the stage, let $S(\PHI)$ be the set of all satisfying assignments of $\PHI$.
Assuming $S(\PHI)\neq\emptyset$, let
\begin{align}\label{eqGibbs}
\mu_{\PHI}(\sigma)&=\vecone\cbc{\sigma\in S(\PHI)}/Z(\PHI)&(\sigma\in\{\pm1\}^{\{x_1,\ldots,x_n\}})
\end{align}
be the uniform distribution on $S(\PHI)$.
Further, write $\SIGMA$ for a sample from $\mu_{\PHI}$.
Then for a satisfying assignment $\tau\in S(\PHI)$ and $\ell\geq1$ the conditional distribution
$\mu_{\PHI}(\nix\mid\SIGMA_{\partial^{2\ell}x_1}=\tau_{\partial^{2\ell}x_1})=
\mu_{\PHI}(\nix\mid \forall y\in\partial^{2\ell}x_1:\SIGMA_{y}=\tau_{y})
$
imposes the `boundary condition' $\tau$ on all variables $y$ at distance $2\ell$ from $x_1$.
The following theorem shows that Belief Propagation does not just approximate the plain, unconditional marginals well \whp, but even the conditional marginals given any conceivable boundary condition.
Recall that $\pr\brk{Z(\PHI)>0}=1-o(1)$ for $d<2$.

\begin{theorem}\label{Thm_BP}
If $d<2$, then
\begin{align}\label{eqThm_BP}
\lim_{\ell\to\infty}\limsup_{n\to\infty}\ \Erw\brk{\max_{\tau\in S(\PHI)}
		\abs{\mu_{\PHI}(\SIGMA_{x_1}=1\mid\SIGMA_{\partial^{2\ell}x_1}=\tau_{\partial^{2\ell}x_1})
			-\nu^{(\ell)}_{\PHI,x_1}(1)}\,\big|\, Z(\PHI)>0}&=0.
\end{align}
\end{theorem}

Since  $\nu^{(\ell)}_{\PHI,x_1}$ does not depend on $\tau$, averaging \eqref{eqThm_BP} on the boundary condition $\tau\in S(\PHI)$ yields 
\begin{align}\label{eqMargApx}
\lim_{\ell\to\infty}\limsup_{n\to\infty}\ \Erw
\brk{\abs{\mu_{\PHI}(\SIGMA_{x_1}=\pm1)-\nu^{(\ell)}_{\PHI,x_1}(\pm1)}\mid Z(\PHI)>0}&=0.
\end{align}
Thus, Belief Propagation approximates the unconditional marginal of $x_1$ well in the limit of large $n$ and $\ell$.
Indeed, because the distribution of $\PHI$ is invariant under permutations of the variables $x_1,\ldots,x_n$, \eqref{eqMargApx} implies that the marginals of all but $o(n)$ variables $x_i$ are within $\pm o(1)$ of the Belief Propagation approximation \whp\ 

But thanks to the presence of the boundary condition $\tau$, \Thm~\ref{Thm_BP} leads to further discoveries.
For a start, applying the triangle inequality to \eqref{eqThm_BP} and \eqref{eqMargApx}, we obtain
\begin{align}\label{eqGU}
\lim_{\ell\to\infty}\limsup_{n\to\infty}\ \Erw\brk{\max_{\tau\in S(\PHI)}
		\abs{\mu_{\PHI}(\SIGMA_{x_1}=1\mid\SIGMA_{\partial^{2\ell}x_1}=\tau_{\partial^{2\ell}x_1})
			-\mu_{\PHI}(\SIGMA_{x_1}=1)}\,\big|\, Z(\PHI)>0}&=0.
\end{align}
Thus, no discernible shift of the marginal of $x_1$ is likely to ensue upon imposition of any possible boundary condition $\tau$.
The spatial mixing property~\eqref{eqGU} is colloquially known as {\em Gibbs uniqueness}~\cite{pnas}.
Further, \eqref{eqGU} rules out extensive long-range correlations.
Specifically, for any fixed $\ell$ the first two variables $x_1,x_2$ likely have distance greater than $4\ell$ in $G(\PHI)$.
Therefore, \eqref{eqGU} implies that for all $d<2$,
\begin{align}\label{eqRS}
\lim_{n\to\infty}\ \sum_{s,t\in\{\pm1\}}\Erw\brk{
		\abs{\mu_{\PHI}(\SIGMA_{x_1}=s,\SIGMA_{x_2}=t)
			-\mu_{\PHI}(\SIGMA_{x_1}=s)\cdot\mu_{\PHI}(\SIGMA_{x_2}=t)}\,\big|\, Z(\PHI)>0}&=0.
\end{align}
Thus, the truth values $\SIGMA_{x_1},\SIGMA_{x_2}$ are asymptotically independent.
Of course, once again by permutation invariance, \eqref{eqRS} implies that asymptotic independence extends to all but $o(n^2)$ pairs of variables $x_i,x_j$ \whp{} 
The decorrelation property \eqref{eqRS} is called {\em replica symmetry} in the physics literature~\cite{pnas}.

Finally, we can clarify the combinatorial meaning of the distribution $\pi_d$ from \Thm~\ref{Thm_Z}.
Namely, $\pi_d$ is the limit of the empirical distribution of the marginal probabilities $\mu_{\PHI}(\SIGMA_{x_i}=1)$.

\begin{corollary}\label{Cor_BP}
For any $0<d<2$ the random probability measure
\begin{align}\label{eqempirical}
\pi_{\PHI}&=\frac1n\sum_{i=1}^n\delta_{\mu_{\PHI}(\SIGMA_{x_i}=1)}
\end{align}
 converges to $\pi_d$ weakly in probability.%
 \footnote{That is, for any continuous function $f:[0,1]\to\RR$ we have 
 $\lim_{n\to\infty}\Erw\abs{\int_0^1f(z)\dd\pi_d(z)-\int_0^1f(z)\dd\pi_{\PHI}(z)}=0$.}
\end{corollary}

\noindent
Thus, the stochastic fixed point equation~\eqref{eqDE} that characterises $\pi_d$ simply expresses that the marginal probabilities $\mu_{\PHI}(\SIGMA_{x_i}=1)$ result from the Belief Propagation recurrence~\eqref{eqBPrec}.

\subsection{Preliminaries and notation}
Throughout we denote by $V_n=\{x_1,\ldots,x_n\}$ the variable set of $\PHI_n$.
Generally, given a $2$-SAT formula $\Phi$ we write $V(\Phi)$ for the set of variables and $F(\Phi)$ for the set of clauses.
The bipartite clause/variable-graph $G(\Phi)$ is defined as in \Sec~\ref{Sec_BP}.
For a vertex $v$ of $G(\Phi)$ we let $\partial(\Phi,v)$ be the set of neighbours.
Where $\Phi$ is apparent we just write $\partial v$.
Moreover, $\partial^\ell(\Phi,v)$ or briefly $\partial^\ell v$ stands for the set of vertices at distance exactly $\ell$ from $v$.
Additionally, $\nabla^{\ell}(\Phi,v)$ denotes the sub-formula obtained from $\Phi$ by deleting all clauses and variables at distance greater than $\ell$ from $v$.
This sub-formula may contain clauses of length less than two.
Further, for a clause $a$ and a variable $x$ of $\Phi$ we let $\sign(x,a)=\sign_\Phi(x,a)\in\{\pm1\}$ be the sign with which $x$ appears in $a$.
In addition, we let $S(\Phi)$ be the set of all satisfying assignments of $\Phi$, $Z(\Phi)=|S(\Phi)|$ and, assuming $Z(\Phi)>0$, we let $\mu_\Phi$ be the probability distribution on $\{\pm1\}^{V(\Phi)}$ that induces the uniform distribution on $S(\Phi)$ as in \eqref{eqGibbs}.
Moreover, $\SIGMA_\Phi=(\SIGMA_{\Phi,x})_{x\in V(\Phi)}$ signifies a uniformly random satisfying assignment; we drop $\Phi$ where the reference is apparent.

For any $\Phi$ we set up Belief Propagation as in  \eqref{eqBPinit}--\eqref{eqBPmarg}.
It is well known that Belief Propagation yields the correct marginals if $G(\Phi)$ is a tree.
To be precise, the {\em depth} of $x\in V(\Phi)$ is the maximum distance between $x$ and a leaf of $G(\Phi)$.

\begin{proposition}[{\cite[\Thm~14.1]{MM}}]\label{Fact_BP}
If $G(\Phi)$ is a tree and $x\in V(\Phi)$, then for any $\ell$ greater than or equal to the depth of $x$ we have
$\mu_{\Phi}(\SIGMA_x=\pm1)=\nu^{(\ell)}_{\Phi,x}(\pm1)$.
\end{proposition}

We will encounter the following functions repeatedly.
For $\eps>0$ let $\Lambda_\eps(z)=\log(z\vee\eps)$ be the log function truncated at $\log\eps$.
Moreover, we need the continuous and mutually inverse functions
\begin{align}\label{eqll}
\psi:\RR&\to(0,1),\quad z\mapsto\bc{1+\tanh(z/2)}/2,&
\varphi:(0,1)&\to\RR,\quad p\mapsto\log(p/(1-p)).
\end{align}

Let $\cP(\RR)$ be the set of all Borel probability measures on $\RR$ with the weak topology.
Moreover, for a real $q\geq1$ let $\cW_q(\RR)$ be the set of all $\rho\in\cP(\RR)$ such that $\int_{\RR}|x|^q\dd \rho(x)<\infty$.
We equip this space with the Wasserstein metric
\begin{align}
W_q(\rho, \rho')= \inf\cbc{\bc{\int_{\RR^2}|x-y|^q\dd \gamma(x,y)}^{1/q}: \gamma \text{ is a coupling of } \rho, \rho'},
\end{align}
thereby turning $\cW_q(\RR)$ into a complete separable space~\cite{Boga}.

For $\rho\in\cP(\RR)$ we denote by $\ETA_\rho,\ETA_{\rho,1},\ETA_{\rho,2},\ldots$ random variables with distribution $\rho$.
Similarly, for $\pi\in\cP(0,1)$ we let $\MU_\pi,\MU_{\pi,1},\MU_{\pi,2},\ldots$ be a sequence of random variables with distribution $\pi$.
We also continue to let $\vec d$ be a Poisson variable with mean $d$ and $\vd^+,\vd^-$ Poisson variables with mean $d/2$.
Moreover, $\vs_1,\vs_1',\vs_2,\vs_2',\ldots\in\cbc{\pm1}$ always denote uniformly distributed random variables.
All of these random variables are mutually independent as well as independent of any other sources of randomness.

{\bf\em Finally, from here on we tacitly assume that $0<d<2$.}

\section{Overview}\label{Sec_outline}

\noindent
The proof of \Thm~\ref{Thm_Z} proceeds in four steps.
First we show that the limit $\pi_d$ from \Thm~\ref{Thm_Z} exists.
Subsequently we establish the fact~\eqref{eqThm_BP} that Belief Propagation approximates the conditional marginals well.
This will easily imply the convergence of the empirical marginals~\eqref{eqempirical} to $\pi_d$.
Third, building upon these preparations, we will prove that the truncated mean  $n^{-1}\Erw[\log(Z(\PHI)\vee1)]$ converges to the r.h.s.\ of \eqref{eqThm_Z}.
The truncation is necessary to deal with the (unlikely) event that $Z(\PHI)=0$.
Finally, we will show that $\log(Z(\PHI)\vee1)$ concentrates about its mean to obtain convergence in probability, thus completing the proof of \Thm~\ref{Thm_Z}.

\subsection{Step 1: density evolution}
We begin by verifying that the distribution $\pi_d$ from \Thm~\ref{Thm_Z} is well-defined and that $\pi_d$ satisfies a tail bound.

\begin{proposition}\label{Prop_BP}
The weak limit $\pi_d=\lim_{\ell\to\infty}\BP_d^\ell(\delta_{1/2})$ exists and
 \begin{align}\label{eqProp_uniqueness2}
\Erw\brk{\log^2\frac{\MU_{\pi_d}}{1-\MU_{\pi_d}}}&<\infty.
\end{align}
Moreover, $\MU_{\pi_d}$ and $1-\MU_{\pi_d}$ are identically distributed and
\begin{align}\label{eqBFE}
\Erw\abs{\log\bc{\prod_{i=1}^{\vd^-}\MU_{\pi_d,i}+\prod_{i=1}^{\vd^+}\MU_{\pi_d,i+\vd^-}}}&<\infty,&
\Erw\abs{\log\bc{1-\MU_{\pi_d,1}\MU_{\pi_d,2}}}&<\infty.
\end{align}
\end{proposition}

The proof of \Prop~\ref{Prop_BP}, which we carry out in \Sec~\ref{Sec_Prop_BP}, is based on a contraction argument.
This argument implies that the fixed point iteration converges rapidly to $\pi_d$, a fact that can be exploited to obtain numerical estimates.
The bounds \eqref{eqBFE} ensure that the expectation on the r.h.s.\ of \eqref{eqThm_Z} is well-defined.

\subsection{Step 2: Gibbs uniqueness}
As a next step we verify the Gibbs uniqueness property \eqref{eqGU}.
We proceed by way of analysing a multi-type Galton-Watson tree $\vT$ that mimics the local structure of the graph $G(\PHI)$ upon exploration from variable $x_1$.
The Galton-Watson process has five types: variable nodes and four types of clause nodes $(+1,+1),(+1,-1),(-1,+1),(-1,-1)$.
The root is a variable node $o$.
Moreover, each variable node spawns independent $\Po(d/4)$ numbers of clauses nodes of each of the four types.
Additionally, each clause has a single offspring, which is a variable.
The semantics of the clause types is that the first component indicates whether the parent variable appears in the clause positively or negatively.
The second component indicates whether the child variable appears as a positive or as a negative literal.
Clearly, for $d\leq1$ the tree $\vT$ is finite with probability one, while infinite trees appear with positive probability for $d>1$.

Let $\vT^{(\ell)}$ be the finite tree obtained from $\vT$ by dropping all nodes at distance greater than $\ell$ from the root.
For even $\ell$ it will be convenient to view $\vT^{(\ell)}$ interchangeably as a tree or as a $2$-SAT formula.
In particular, we write $\partial^{2\ell}o=\partial^{2\ell}(\vT,o)$ for the set of all variables at distance exactly $2\ell$ from $o$.
The following proposition, which is the linchpin of the entire proof strategy, establishes the Gibbs uniqueness property for the tree formula $\vT^{(2\ell)}$.

\begin{proposition}\label{Prop_uniqueness}
We have
\begin{align}\label{eqProp_uniqueness1}
\lim_{\ell\to\infty}\Erw\brk{\max_{\tau\in S(\vT^{(2\ell)})}
	\abs{\mu_{\vT^{(2\ell)}}(\SIGMA_o=1\mid\SIGMA_{\partial^{2\ell}o}=\tau_{\partial^{2\ell}o})-\mu_{\vT^{(2\ell)}}(\SIGMA_o=1)}}&=0.
\end{align}
\end{proposition}

\noindent
Thus, \whp{} no conceivable boundary condition is apt to significantly shift the marginal of  the root.

We prove \Prop~\ref{Prop_uniqueness} by a subtle contraction argument in combination with a construction of extremal boundary conditions of the tree formula $\vT^{(2\ell)}$.
More specifically, we will construct boundary conditions $\SIGMA^{\pm}$ that maximise or minimise the conditional probability
\begin{align}\label{eqextremal}
\mu_{\vT^{(2\ell)}}(\SIGMA_o=1\mid\SIGMA_{\partial^{2\ell}o}=\SIGMA_{\partial^{2\ell}o}^{\pm}),
\end{align}
respectively.
Then we will show that the difference of the conditional marginals induced by both these extremal boundary conditions vanishes with probability tending to one as $\ell\to\infty$.
The delicate point is that the extremal boundary conditions $\SIGMA^{\pm}$ depend on the tree $\vT^{(2\ell)}$.
Thus, at first glance it seems that we need to pass the tree twice, once top--down to construct $\SIGMA^{\pm}$ and then bottom--up to calculate the conditional marginals~\eqref{eqextremal}.
But such an analysis seems untenable because after the top--down pass the tree is exposed and `no randomness remains' to facilitate the bottom--up phase.
Fortunately, we will see that a single stochastic fixed point equation captures both the top--down and the bottom--up phase.
This discovery reduces the proof of \Prop~\ref{Prop_uniqueness} to showing that the fixed point iteration contracts. 
The details of this delicate argument can be found in \Sec~\ref{Sec_Prop_uniqueness}.

\Prop~\ref{Prop_uniqueness} easily implies the Gibbs uniqueness condition \eqref{eqGU} and thereby \Thm~\ref{Thm_BP}. 
A further consequence is the asymptotic independence of the joint truth values of bounded numbers of variables.

\begin{corollary}\label{Cor_uniqueness}
The statement \eqref{eqThm_BP} is true and for any integer $k\geq2$ we have
\begin{align*}
\lim_{n\to\infty}
\sum_{\sigma\in\{\pm1\}^k}
\Erw\brk{\abs{\mu_{\PHI}(\SIGMA_{x_1}=\sigma_1,\ldots,\SIGMA_{x_k}=\sigma_k)-
	\prod_{i=1}^k\mu_{\PHI}(\SIGMA_{x_i}=\sigma_i)}\mid Z(\PHI)>0}&=0.
\end{align*}
\end{corollary}

\subsection{Step 3: the Aizenman-Sims-Starr scheme}
The aforementioned results pave the way for deriving an expression for the conditional expectation of $\log Z(\PHI)$ given that $\PHI$ is satisfiable.
Since $\PHI$ is satisfiable \whp\ for all $d<2$, an equivalent task is to calculate $\Erw[\log(Z(\PHI)\vee1)]$.
To this end we seize upon a simple but powerful strategy colloquially called the Aizenman-Sims-Starr scheme~\cite{Aizenman}.
Originally proposed in the context of the Sherrington-Kirkpatrick spin glass model, this proof strategy suggests to compute the asymptotic mean of a random variable on a `system' of size $n$ by carefully estimating the change of that mean upon going to a `system' of size $n+1$.
This difference is calculated by coupling the systems of size $n$ and $n+1$ such that the latter is obtained from the former by a small expected number of local changes.

We apply this idea to the random $2$-SAT problem by coupling the random formula $\PHI_n$ with $n$ variables and $\Po(dn/2)$ clauses and the random formula $\PHI_{n+1}$ with $n+1$ variables and $\Po(d(n+1)/2)$ clauses.
Roughly speaking, we obtain $\PHI_{n+1}$ from $\PHI_n$ by adding a new variable $x_{n+1}$ along with a few random adjacent clauses that connect $x_{n+1}$ with the variables $x_1,\ldots,x_n$ of $\PHI_n$.
Then the information about the joint distribution of the truth values of bounded numbers of variables furnished by  Corollaries~\ref{Cor_BP} and~\ref{Cor_uniqueness} and the tail bound~\eqref{eqProp_uniqueness2} will enable us to accurately estimate $\Erw\brk{\log(Z(\PHI_{n+1})\vee 1) - \log(Z(\PHI_{n})\vee 1)}$.

Needless to say, upon closer inspection matters will emerge to be rather subtle.
The main source of complications is that, in contrast to other models in mathematical physics such as the Sherrington-Kirkpatrick model or the Ising model, the 2-SAT problem has hard constraints.
Thus, the addition of a single clause could trigger a dramatic drop in the partition function.
In fact, in the worst case a single awkward clause could wipe out all satisfying assignments.
In \Sec~\ref{Sec_Prop_Bethe} we will iron out all these difficulties and prove the following.

\begin{proposition}\label{Prop_Bethe}
We have
\begin{align}\label{eqProp_Bethe_stmt}
\lim_{n\to\infty}\Erw[\log(Z(\PHI_{n+1})\vee1)]-\Erw[\log(Z(\PHI_{n})\vee1)]=
\Erw\brk{\log\bc{\prod_{i=1}^{\vd^-}\MU_{\pi_d,i}+\prod_{i=1}^{\vd^+}\MU_{\pi_d,i+\vd^-}}
-\frac d2\log\bc{1-\MU_{\pi_d,1}\MU_{\pi_d,2}}}.
\end{align}
\end{proposition}

We notice that \eqref{eqBFE} guarantees that the r.h.s.\ of \eqref{eqProp_Bethe_stmt} is well-defined.
As an immediate consequence of \Prop~\ref{Prop_Bethe} we obtain a formula for $\Erw[\log (Z(\PHI)\vee 1)]$.

\begin{corollary}\label{Cor_Bethe}
For any $d<2$ we have $$\lim_{n\to\infty}\frac1n\Erw[\log(Z(\PHI)\vee1)]=\Erw\brk{\log\bc{\prod_{i=1}^{\vd^-}\MU_{\pi_d,i}+\prod_{i=1}^{\vd^+}\MU_{\pi_d,i+\vd^-}}
-\frac d2\log\bc{1-\MU_{\pi_d,1}\MU_{\pi_d,2}}}
.$$
\end{corollary}
\begin{proof}
Writing $\Erw[\log(Z(\PHI)\vee1)]$ as a telescoping sum and applying \Prop~\ref{Prop_Bethe}, we obtain
\begin{align*}
\lim_{n\to\infty}\frac{1}{n}\Erw[\log(Z(\PHI_n)\vee1)]&=
\lim_{n\to\infty}\frac1n\sum_{N=2}^{n-1}\Erw[\log(Z(\PHI_{N+1})\vee1)]-\Erw[\log(Z(\PHI_{N})\vee1)]\\
&=\Erw\brk{\log\bc{\prod_{i=1}^{\vd^-}\MU_{\pi_d,i}+\prod_{i=1}^{\vd^+}\MU_{\pi_d,i+\vd^-}}
-\frac d2\log\bc{1-\MU_{\pi_d,1}\MU_{\pi_d,2}}},
\end{align*}
as desired.
\end{proof}

\subsection{Step 4: concentration}
The final step towards \Thm~\ref{Thm_Z} is to show that $\log(Z(\PHI)\vee1)$ concentrates about its mean.

\begin{proposition}\label{Prop_conc}
We have $\lim_{n\to\infty}n^{-1}\Erw\abs{\log(Z(\PHI) \vee 1)-\Erw[\log(Z(\PHI)\vee 1)]}=0$.
\end{proposition}

\Prop~\ref{Prop_conc} does not easily follow from routine arguments such as the Azuma-Hoeffding inequality.
Once more the issue is that changing a single clause could alter $\log(Z(\PHI)\vee 1)$ by as much as $\Theta(n)$.
Instead we will resort to another technique from mathematical physics called the interpolation method.
The details can be found in \Sec~\ref{Sec_Prop_conc}.

\begin{proof}[Proof of \Thm~\ref{Thm_Z}]
The theorem follows from \Prop~\ref{Prop_BP}, \Cor~\ref{Cor_Bethe} and \Prop~\ref{Prop_conc}.
\end{proof}

\section{Discussion}\label{Sec_related}

\noindent
The random $2$-SAT satisfiability threshold was established mathematically shortly after the experimental work of Cheeseman, Kanefsky and~Taylor \cite{Cheeseman} that triggered the quest for satisfiability thresholds appeared.
The second successful example, nearly a decade later, was the random $1$-in-$k$-SAT threshold (to satisfy exactly one literal in each clause), which Achlioptas,  Chtcherba, Istrate and Moore pinpointed by analysing the Unit Clause algorithm~\cite{ACIM}.
In a subsequent landmark contribution Dubois and Mandler determined the $3$-XORSAT threshold via the second moment method~\cite{DuboisMandler}.
Subsequent work extended this result to random $k$-XORSAT~\cite{Cuckoo,PittelSorkin}.
Finally, the most notable success thus far has been the verification of the `1RSB cavity method' prediction~\cite{MPZ} of the random $k$-SAT threshold for large $k$ due to Ding, Sly and Sun~\cite{DSS3}, the culmination of a line of work that refined the use of the second moment method~\cite{nae,yuval,KostaSAT}.

Over the past two decades the general theme of estimating the partition functions of discrete structures has received a great deal of attention; e.g., \cite{Barvinok}.
With respect to random $2$-SAT (and, more generally, $k$-SAT), Montanari and Shah~\cite{MS}, Panchenko~\cite{Panchenko2} and  Talagrand~\cite{Talagrand} investigated `soft' versions of the partition function.
To be precise, introducing a parameter $\beta>0$ called the `inverse  temperature', these articles study the random variable
\begin{align}\label{eqZbeta}
Z_\beta(\PHI)&=\sum_{\sigma\in\{\pm1\}^n}\prod_{i=1}^{\vm}\exp\bc{-\beta\vecone\cbc{\sigma\mbox{ violates clause }a_i}}.
\end{align}
Thus, instead of dismissing assignments that fail to satisfy all clauses outright, there is an $\exp(-\beta)$ penalty factor for each violated clause.
Talagrand~\cite{Talagrand} computes $\lim_{n\to\infty}n^{-1}\Erw[\log Z_\beta(\PHI)]$ for $\beta$ not exceeding a small but unspecified $\beta_0>0$.
Panchenko~\cite{Panchenko2} calculates this limit under the assumption $(4\beta\wedge1) d<1$.
Thus, for $\beta>1/4$ the result is confined to $d<1$, in which case the random graph $G(\PHI)$ is sub-critical and both $Z_\beta(\PHI)$ and the actual number $Z(\PHI)$ of satisfying assignments could be calculated via elementary methods.
Furthermore, Montanari and Shah~\cite{MS} obtain $\lim_{n\to\infty}n^{-1}\Erw[\log Z_\beta(\PHI)]$ for all finite $\beta$ under the assumption $d<1.16\dots\enspace$.
Although for any fixed formula $\PHI$ the limit $\lim_{\beta\to\infty}Z_\beta(\PHI)$ is equal to the number of satisfying assignments, it is not possible to interchange the limits $\beta\to\infty$ and $n\to\infty$.
Thus, \cite{MS,Panchenko2} do not yield the the number of actual  satisfying assignments even for $d<1.16\dots$ or $d<1$, respectively.
Apart from estimating $\Erw\log Z_\beta(\PHI)$, Montanari and Shah~\cite{MS} also show that the Belief Propagation message passing scheme approximates the marginals of the Boltzmann distribution that goes with $Z_\beta(\PHI)$ well, i.e., they obtain a  `soft' version of \Thm~\ref{Thm_BP} for $d<1.16\dots\enspace$.

In terms of proof techniques, all three contributions~\cite{MS,Panchenko2,Talagrand} are based on establishing the Gibbs uniqueness property.
So is the present paper.
But while \cite{MS,Panchenko2,Talagrand} rely on relatively straightforward contraction arguments,
 a key distinction is that here we develop a more accurate (and delicate) method for verifying the Gibbs uniqueness property based on the explicit construction of an extremal boundary condition.
This is the key to pushing the range of $d$ all the way up to the satisfiability threshold $d=2$.

Specifically, in order to construct a boundary condition of the random tree $\vT^{(2\ell)}$ for large $\ell$ that maximises the conditional probability of observing the truth value $+1$ at the root we will work our way top--down from the root to level $2\ell$.
Exposing the degrees and the signs with which the variables appear, the construction assigns a `desired' truth value to each variable of the tree so as to nudge the parent variable towards its desired value as much as possible.
Subsequently, once this process reaches the bottom level of the tree, we go into reverse gear and study the Belief Propagation messages bottom--up to calculate the conditional marginal of the root.
Clearly, analysing this upwards process seems like a tall order because the tree was already exposed during the top-down phase, a challenge that is exacerbated by the presence of hard constraints.
Fortunately, in \Sec~\ref{Sec_Prop_uniqueness} we will see how this problem can be transformed into the study of another stochastic fixed point equation that captures the effect of the children's `nudging' their parents.
This fixed point problem is amenable to the contraction method.
A spatial mixing analysis from an extremal boundary condition was previously conducted in by Dembo and Montanari~\cite{DM} for the Ising model on random graphs.
But of course a crucial difference is that in the Ising model the extremal boundary conditions are constant (all-$+1$ and all-$-1$, respectively).

A second novelty of the present work is that we directly deal with the `hard' $2$-SAT problem.
 Montanari and Shah~\cite{MS} interpolate on the `inverse temperature' parameter $\beta>0$, effectively working their way from smaller to larger $\beta$.
Because the limits $\beta\to\infty$ and $n\to\infty$ do not commute, this approach does not seem applicable to problems with hard constraints.
Furthermore, while Panchenko~\cite{Panchenko,Panchenko2} applies the Aizenman-Sims-Starr scheme to the soft constraint version, the hard problem of counting actual satisfying assignments requires a far more careful analysis.
Indeed, adding one clause can shift $\log Z_\beta(\PHI)$ merely by $\pm\beta$.
By contrast, a single additional clause could very well reduce the logarithm $\log Z(\PHI)$ of the number of satisfying assignments by as much as $\Omega(n)$, or even render the formula unsatisfiable.
A few prior applications of the Aizenman-Sims-Starr scheme to problems with hard constraints exist~\cite{Ayre,CKM,CKPZ}, but these hinge on peculiar symmetry properties that enable an indirect approach via a `planted' version of the problem in question.
The required symmetries for this approach are absent in several important problems, with random satisfiability the most prominent example.
Thus, a significant technical contribution of the present work is that we show how to apply the Aizenman-Sims-Starr scheme directly to problems with hard constraints.
Among other things, this requires a careful quantification of the probabilities of rare, potentially cataclysmic events in comparison to their impact on $\log Z(\PHI)$.
That said, we should point out that~\cite{MS,Panchenko2,Talagrand} actually also deal with the (soft) $k$-SAT partition function for $k>2$ for certain regimes of clause/variable densities, while the technique that we develop here does not seem to  extend beyond binary problems.

A mathematical physics technique called the interpolation method, first proposed by Guerra for the study of the Sherrington-Kirkpatrick model~\cite{Guerra}, can be applied to the random $k$-SAT problem~\cite{FranzLeone,PanchenkoTalagrand} to bound the number of satisfying assignments from above.
For $k=2$ the interpolation method yields the upper bound
\begin{align}\label{eqInterpolation}
\frac{1}{n}\log Z(\PHI)&\leq\inf_{\pi\in\cP(0,1)}
\Erw\brk{\log\bc{\prod_{i=1}^{\vd^-}\MU_{\pi,i}+\prod_{i=1}^{\vd^+}\MU_{\pi,i+\vd^-}}
-\frac d2\log\bc{1-\MU_{\pi,1}\MU_{\pi,2}}}+o(1)&&\mbox\whp,
\end{align}
for all $0<d<2$; we will revisit this bound in \Sec~\ref{Sec_Prop_conc}.
Since the expression on the r.h.s.\ coincides with \eqref{eqThm_Z} for $\pi=\pi_d$, the main contribution of \Thm~\ref{Thm_Z} is the matching lower bound on $\log Z(\PHI)$.
Furthermore, Abbe and Montanari~\cite{AM} used the interpolation method to establish the {\em existence} of a  function $\phi$ such that 
\begin{align}\label{eqAM}
\lim_{n\to\infty}n^{-1}\log(Z(\PHI)\vee 1)&=\phi(d)&&\mbox{in probability}
\end{align}
for all but a countable number of $d\in(0,2)$.
\Thm~\ref{Thm_Z} actually determines $\phi(d)$ and shows that convergence holds for {\em all} $d\in(0,2)$.
Clearly, \eqref{eqAM} implies the concentration bound from \Prop~\ref{Prop_conc} for all $d$ outside the countable set.
But of course we need concentration for all $d$, and
in \Sec~\ref{Sec_Prop_conc} we will use the upper bound \eqref{eqInterpolation} to prove this concentration result.
As an aside, a conditional concentration inequality for $\log Z(\PHI)$, quoted in~\cite{Lalo}, was obtained by Sharell~\cite{Sharell} (unpublished).
But the necessary conditions  appear to be difficult to check.

In addition, several prior contributions deal with the combinatorial problem of counting solutions to random CSPs.
For problems such as $k$-NAESAT, $k$-XORSAT or graph colouring where the first moment provides the correct answer due to inherent symmetry properties, the second moment method and small subgraph conditioning yield very precise information as to the number of solutions~\cite{CKM,COW,Feli2}.
Verifying that the number of solutions is determined by the physicists' 1RSB formula~\cite{MM}, the contribution of Sly, Sun and Zhang~\cite{SSZ} on the random regular $k$-NAESAT problem near its satisfiability threshold~\cite{DSS1} deals with an even more intricate scenario.

Finally, returning to random $2$-SAT, as an intriguing question for future work determining the precise limiting distribution of $\log Z(\PHI)$ stands out.
This random variable has standard deviation $\Omega(\sqrt n)$ for all $0<d<2$ even once we condition on $\vm$, as is easily seen by re-randomising the signs of the literals in small components.
In effect, $\log Z(\PHI)$ is far less concentrated than the partition functions of symmetric random constraint satisfaction problems~\cite{CKM}.
May $n^{-1/2}(\log Z(\PHI)-\Erw[\log Z(\PHI)])$ be asymptotically normal?

\section{Proof of \Prop~\ref{Prop_BP}}\label{Sec_Prop_BP}

\noindent
We prove \Prop~\ref{Prop_BP} by means of a contraction argument.
The starting point is the following observation.
For $\ell\geq 0$ let $\pi_d^{(\ell)}=\DE_d^\ell(\delta_{1/2})$ be the probability measure obtained after $\ell$  iterations of the operator $\DE_d(\nix)$.

\begin{fact}\label{Lemma_sym}
For all $\ell\geq0$ the random variables $\MU_{\pi_d^{(\ell)}}$ and $1-\MU_{\pi_d^{(\ell)}}$ are identically distributed.
\end{fact}
\begin{proof}
 This is because $\vd^-,\vd^+$ and hence the random variables 
\begin{align*}
\bc{\prod_{i=1}^{\vd^-}\MU_{\pi^{(\ell-1)}_d,i}, 
\prod_{i=1}^{\vd^-}\MU_{\pi_d^{(\ell-1)},i}+\prod_{i=1}^{\vd^+}\MU_{\pi_d^{(\ell-1)},i+\vd^-}} \quad \text{and} \quad \bc{\prod_{i=1}^{\vd^+}\MU_{\pi_d^{(\ell-1)},i+\vd^-}, 
\prod_{i=1}^{\vd^-}\MU_{\pi_d^{(\ell-1)},i}+\prod_{i=1}^{\vd^+}\MU_{\pi_d^{(\ell-1)},i+\vd^-}}
\end{align*}
from \eqref{eqdensityEv} are identically distributed.
\end{proof}

Due to Fact~\ref{Lemma_sym} we can rewrite the construction of the sequence $\pi_d^{(\ell)}$ in terms of another operator that is easier to analyse.
This operator describes the expression \eqref{eqdensityEv} in terms of log-likelihood ratios, a simple reformulation that proved useful in the context of Belief Propagation for random satisfiability before~\cite{Allerton}.
Thus,  we define an operator $\LDE_d:\cP(\RR)\to\cP(\RR)$, $\rho\mapsto\hat\rho$ by letting $\hat\rho$ be the distribution of the random variable
\begin{align}\label{Eq_BTreeOperator}
\sum_{i=1}^{\vd}
		\vs_i\log\frac{1+\vs_i'\tanh(\ETA_{\rho,i}/2)}2. 
\end{align}
Further, let $\rho_d^{(\ell)}=\LDE_d^\ell(\delta_0)\in\cP(\RR)$ be the result of $\ell$ iterations of  $\LDE_d$ launched from the atom at zero.
We recall the functions $\psi,\varphi$ from \eqref{eqll}. 
For a measure $\rho \in \cP(\RR)$ and a measurable $f: \RR \to \RR$ let $f(\rho)$ denote the pushforward measure of $\rho$ that assigns mass $\rho(f^{-1}(A))$ to Borel sets $A \subset \RR$.

\begin{lemma}\label{Lemma_loglikelihood}
For all $\ell\geq0$ we have $\pi_d^{(\ell)}=\psi(\rho_d^{(\ell)})$.
\end{lemma}
\begin{proof}
Since $\psi(\delta_0)=\delta_{1/2}$, the assertion is true for $\ell=0$.
Proceeding by induction, we obtain
\begin{align}\nonumber
\MU_{\pi_d^{(\ell+1)}}&\disteq
\frac{\prod_{i=1}^{\vd^+}\MU_{\pi_d^{(\ell)},i}}
{\prod_{i=1}^{\vd^-}\MU_{\pi_d^{(\ell)},i}+\prod_{i=1}^{\vd^+}\MU_{\pi_d^{(\ell)},i+\vd^-}}
=\psi\bc{\log\frac{\prod_{i=1}^{\vd^-}\MU_{\pi_d^{(\ell)},i}}{\prod_{i=1}^{\vd^+}\MU_{\pi_d^{(\ell)},i+\vd^-}}}\\
&\label{eqLemma_loglikelihood1}
=\psi\bc{\sum_{i=1}^{\vd^-}\log\bc{\MU_{\pi_d^{(\ell)},i}}-
	\sum_{i=1}^{\vd^+}\log\bc{\MU_{\pi_d^{(\ell)},i+\vd^-}}}
\disteq
\psi\bc{\sum_{i=1}^{\vd}\vs_i\log\MU_{\pi_d^{(\ell)},i}}
\disteq
\psi\bc{\sum_{i=1}^{\vd}\vs_i\log\bc{\psi(\ETA_{\rho_d^{(\ell)},i})}}.
\end{align}
Moreover, since $\vs_i\in\{\pm1\}$ is random, it is immediate from (\ref{Eq_BTreeOperator}) that 
$\ETA_{\rho_d^{(\ell)},i}\disteq-\ETA_{\rho_d^{(\ell)},i}$.
Consequently, \eqref{eqLemma_loglikelihood1} yields
\begin{align*}
\MU_{\pi_d^{(\ell+1)}}&\disteq\psi\bc{\sum_{i=1}^{\vd}\vs_i\log
	\bc{\psi(\vs_i'\ETA_{\rho_d^{(\ell)},i})}}\disteq\psi(\ETA_{\rho_d^{(\ell+1)}}),
\end{align*}
which completes the induction.
\end{proof}

Due to the continuous mapping theorem, to establish convergence of $(\pi^{(\ell)}_d)_{\ell\geq 0}$ it suffices to show that $(\rho_d^{(\ell)})_{\ell\geq 0}$ converges weakly.
To this end, we will prove that the operator $\LDE_d(\nix)$ is a contraction.

\begin{lemma}\label{Lemma_BP1}
If $d<2$, then $\LDE_d$ is a contraction on the space $\cW_2(\RR)$.
\end{lemma}
\begin{proof}
The operator $\LDE_d$ maps the space  $\cW_2(\RR)$ into itself because the derivative of $x\mapsto\log((1+\tanh(x/2))/2)$ is bounded by one in absolute value for all $x\in\RR$.
To show contraction let $\rho,\rho'\in\cW_2(\RR)$ and consider a sequence of independent random pairs $(\ETA_i,\ETA_i')_{i\geq1}$ such that the $\ETA_i$ have distribution $\rho$ and the $\ETA_i'$ have distribution $\rho'$.
Because the signs $\vs_i$ are uniform and independent, we obtain
\begin{align}\nonumber
W_2(\LDE(\rho),\LDE(\rho'))^2&\leq
\Erw\brk{\bc{\sum_{i=1}^{\vd}\vs_i\log\frac{1+\vs_i'\tanh(\ETA_i/2)}{1+\vs_i'\tanh(\ETA_i'/2)}}^2}
=\Erw\brk{\sum_{h,i=1}^{\vd}\vs_h\vs_i\log\frac{1+\vs_h'\tanh(\ETA_h/2)}{1+\vs_h'\tanh(\ETA_h'/2)}\log\frac{1+\vs_i'\tanh(\ETA_i/2)}{1+\vs_i'\tanh(\ETA_i'/2)}}\\
&=\Erw\brk{\sum_{i=1}^{\vd}\log^2\frac{1+\vs_i'\tanh(\ETA_i/2)}{1+\vs_i'\tanh(\ETA_i'/2)}}
=d\Erw\brk{\log^2\frac{1+\vs_1\tanh(\ETA_1/2)}{1+\vs_1\tanh(\ETA_1'/2)}}.
\label{eqLemma_BP1_1}
\end{align}
Further,
\begin{align}
\log^2\frac{1+\tanh(\ETA_1/2)}{1+\tanh(\ETA_1'/2)}&=
\brk{\int_{\ETA_1'}^{\ETA_1}\frac{\partial\log(1+\tanh(z/2))}{\partial z}\dd z}^2
=\brk{\int_{\ETA_1\wedge \ETA_1'}^{\ETA_1\vee\ETA_1'}\frac{1-\tanh(z/2)}{2}\dd z}^2,\label{eqLemma_BP1_2}\\
\log^2\frac{1-\tanh(\ETA_1/2)}{1-\tanh(\ETA_1'/2)}&=
\brk{\int_{\ETA_1'}^{\ETA_1}\frac{\partial\log(1-\tanh(z/2))}{\partial z}\dd z}^2
=\brk{\int_{\ETA_1\wedge \ETA_1'}^{\ETA_1\vee\ETA_1'}\frac{1+\tanh(z/2)}{2}\dd z}^2.\label{eqLemma_BP1_3}
\end{align}
Combining \eqref{eqLemma_BP1_2}--\eqref{eqLemma_BP1_3} and applying the Cauchy-Schwarz inequality, we obtain
\begin{align}\nonumber
\Erw\brk{\log^2\frac{1+\vs_1\tanh(\ETA_1/2)}{1+\vs_1\tanh(\ETA_1'/2)}}&=
\frac12\Erw\brk{\brk{\int_{\ETA_1\wedge \ETA_1'}^{\ETA_1\vee\ETA_1'}\frac{1-\tanh(z/2)}{2}\dd z}^2+
\brk{\int_{\ETA_1\wedge \ETA_1'}^{\ETA_1\vee\ETA_1'}\frac{1+\tanh(z/2)}{2}\dd z}^2}\\
&\leq\frac12\Erw\brk{\abs{\ETA_1-\ETA_1'}\int_{\ETA_1\wedge \ETA_1'}^{\ETA_1\vee\ETA_1'}
\bcfr{1-\tanh(z/2)}{2}^2+\bcfr{1+\tanh(z/2)}{2}^2\dd z}\leq\frac12\Erw\brk{\bc{\ETA_1-\ETA_1'}^2}.
\label{eqLemma_BP1_4}
\end{align}
Finally, \eqref{eqLemma_BP1_1} and \eqref{eqLemma_BP1_4} yield
$W_2(\LDE(\rho),\LDE(\rho'))^2\leq d\Erw[(\ETA_1-\ETA_1')^2]/2$,
which implies contraction because $d<2$.
\end{proof}

\begin{proof}[Proof of \Prop~\ref{Prop_BP}]
Together with the Banach fixed point theorem \Lem~\ref{Lemma_BP1} ensures that the $W_2$-limit
$\rho_d=\lim_{\ell\to\infty}\LDE_d^\ell(\delta_0)$
exists.
Therefore, \Lem~\ref{Lemma_loglikelihood} implies that the sequence $(\pi_d^{(\ell)})_{\ell\geq0}$ converges weakly.
In addition, since $\rho_d\in\cW_2(\RR)$, \Lem~\ref{Lemma_loglikelihood} also implies the bound \eqref{eqProp_uniqueness2}.
Finally, to prove~\eqref{eqBFE} we apply \eqref{eqProp_uniqueness2} to obtain
\begin{align*}
\Erw\abs{\log\bc{\prod_{i=1}^{\vd^-}\MU_{\pi_d,i}+\prod_{i=1}^{\vd^+}\MU_{\pi_d,i+\vd^-}}}
	&\leq\log(2)	-\Erw\log{\prod_{i=1}^{\vd^-}\MU_{\pi_d,i}}\leq\log(2)-\frac d2\Erw\log\MU_{\pi_d,1}\leq
		2\log(2)+d\Erw\abs{\log\frac{\MU_{\pi_d}}{1-\MU_{\pi_d}}}<\infty,\\
\Erw\abs{\log(1-\MU_{\pi_d,1}\MU_{\pi_d,2})}&\leq
\Erw\abs{\log(1-\MU_{\pi_d})}\leq
\Erw\abs{\log\frac{\MU_{\pi_d}}{1-\MU_{\pi_d}}}+\log2<\infty,
\end{align*}
thereby completing the proof.
\end{proof}

\section{Proof of \Prop~\ref{Prop_uniqueness}}\label{Sec_Prop_uniqueness}

\subsection{Outline}
The goal is to prove that the marginal of the root variable $o$ of $\vT^{(2\ell)}$ remains asymptotically invariant even upon imposition of an arbitrary (feasible) boundary condition on the variables at distance $2\ell$ from the root $o$.
A priori, a proof of this statement seems challenging because of the very large number of possible boundary conditions.
Indeed, we expect about $d^\ell$ variables at distance $2\ell$. 
But a crucial feature of the $2$-SAT problem is that we can construct a pair of extremal boundary conditions.
One of these maximises the probability that the root is set to one.
The other one minimises that probability.
As a consequence, instead of inspecting all possible boundary conditions, it suffices to show that the marginals on the root $o$ that these two extremal boundary induce asymptotically coincide with the unconditional marginals.
Of course, due to symmetry it actually suffices to consider the `positive' extremal boundary condition that maximally nudges the root towards $+1$.

\begin{figure}
  \centering
    \begin{tikzpicture}[every node/.style = {draw, align=center, minimum width=.6cm, minimum height=.6cm}]
  \node[shape=circle, fill=gray!30] (root) at (0,0) {$+1$};
  \node[shape=rectangle] (cl11) at (-1,-.75) {$ $};
  \node[shape=rectangle] (cl12) at (1,-.75) {$ $};
  \path [-, blue!70!black] (root) edge node[draw=none, pos=.1, left] {$+$} (cl11);
  \path [-, red!70!black] (root) edge node[draw=none, pos=.2, right] {$-$} (cl12);
  
  \node[shape=circle, fill=blue!10] (v11) at (-2,-1.5) {$+1$};
  \node[shape=circle, fill=blue!10] (v12) at (2,-1.5) {$+1$};
  \path [-, red!70!black] (v11) edge node[draw=none, pos=.3, above] {$-$} (cl11);
  \path [-, blue!70!black] (v12) edge node[draw=none, pos=.3, above] {$+$} (cl12);
  
  \node[shape=rectangle] (cl21) at (-2.5,-3) {$ $};
  \node[shape=rectangle] (cl22) at (-1.5,-3) {$ $};
  \path [-, blue!70!black] (v11) edge node[draw=none, pos=.5, left] {$+$} (cl21);
  \path [-, red!70!black] (v11) edge node[draw=none, pos=.5, right] {$-$} (cl22);
  
  \node[shape=rectangle] (cl23) at (1,-3) {$ $};
  \node[shape=rectangle] (cl24) at (2,-3) {$ $};
  \node[shape=rectangle] (cl25) at (3,-3) {$ $};
  \path [-, blue!70!black] (v12) edge node[draw=none, pos=.5, left] {$+$} (cl23);
  \path [-, red!70!black] (v12) edge node[draw=none, pos=.5, left] {$-$} (cl24);
  \path [-, red!70!black] (v12) edge node[draw=none, pos=.5, right] {$-$} (cl25);
  
  \node[shape=circle, fill=orange!30] (v21) at (-2.5,-4.5) {$-1$};
  \node[shape=circle, fill=orange!30] (v22) at (-1.5,-4.5) {$-1$};
  \path [-, red!70!black] (v21) edge node[draw=none, pos=.5, left] {$+$} (cl21);
  \path [-, blue!70!black] (v22) edge node[draw=none, pos=.5, right] {$-$} (cl22);
  
  \node[shape=circle, fill=orange!30] (v23) at (1,-4.5) {$-1$};
  \node[shape=circle, fill=orange!30] (v24) at (2,-4.5) {$+1$};
  \node[shape=circle, fill=orange!30] (v25) at (3,-4.5) {$-1$};
  \path [-, red!70!black] (v23) edge node[draw=none, pos=.5, left] {$+$} (cl23);
  \path [-, blue!70!black] (v24) edge node[draw=none, pos=.5, left] {$+$} (cl24);
  \path [-, blue!70!black] (v25) edge node[draw=none, pos=.5, right] {$-$} (cl25);
  
\end{tikzpicture} 
\caption{The graph $G(\PHI)$ together with extremal boundary condition $\SIGMA^+$. Variables are indicated by circles and clauses by squares. The labels on the edges illustrate the sign with which  variables appears in the clauses. To obtain the extremal boundary condition $\SIGMA^+$ we proceed top-down. The truth values of the children are chosen so as to nudge the parent variables in the direction provided by $\SIGMA^+$.}\label{Fig_Max}
\end{figure}
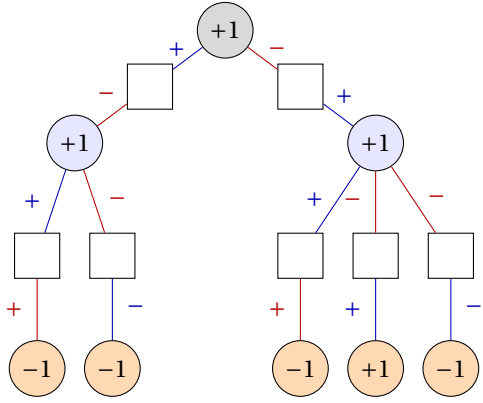

To construct this extremal boundary condition we define a satisfying assignment $\SIGMA^{+}$ by working our way down the tree $\vT^{(2\ell)}$.
We begin by defining $\SIGMA^{+}_o=1$.
Further, suppose for $\ell\geq1$ the values of the variables at distance $2(\ell-1)$ from $o$ have been defined already.
Consider a variable $v\in\partial^{2\ell}o$, its parent clause $a$ and the parent variable $u$ of $a$.
Our aim is to choose $\SIGMA_v^{+}$ so as to `nudge' $u$ towards $\SIGMA_u^{+}$ as much as possible.
To this end we set $\SIGMA^{+}_v$ so as to not satisfy $a$ if setting $u$ to $\SIGMA_u^{+}$ satisfies $a$.
Otherwise we pick the value that satisfies $a$; see Figure~\ref{Fig_Max}.
In formulas,
\begin{align*}
\SIGMA^{+}_v&=\sign(a,v)\vecone\{\sign(a,u)\neq\SIGMA^{+}_u\}-\sign(a,v)\vecone\{\sign(a,u)=\SIGMA^{+}_u\}.
\end{align*}
The following lemma verifies that $\SIGMA^{+}$ is extremal, i.e., that imposing the values provided by $\SIGMA^{+}$ on the boundary variables $\partial^{2\ell}o$ maximises the probability of the truth value $1$ at the root $o$.
The proof can be found in \Sec~\ref{Sec_Lemma_extremal}.

\begin{lemma}\label{Lemma_extremal}
For any integer $\ell\geq0$ we have
$ \max_{\tau\in S(\vT^{(2\ell)})}\mu_{\vT^{(2\ell)}}(\SIGMA_o=1\mid\SIGMA_{\partial^{2\ell}o}=\tau_{\partial^{2\ell}o})= \mu_{\vT^{(2\ell)}}(\SIGMA_o=1\mid\SIGMA_{\partial^{2\ell}o}=\SIGMA^{+}_{\partial^{2\ell}o}).$
\end{lemma}

\Lem~\ref{Lemma_extremal} reduces the task of proving \Prop~\ref{Prop_uniqueness} to establishing the following statement.

\begin{proposition}\label{Lemma_condMarg}
We have
$\lim_{\ell\to\infty}\Erw\abs{\mu_{\vT^{(2\ell)}}(\SIGMA_o=1)- \mu_{\vT^{(2\ell)}}(\SIGMA_o=1\mid\SIGMA_{\partial^{2\ell}o}=\SIGMA^{+}_{\partial^{2\ell}o})}=0.$
\end{proposition}

 \noindent
 In words, the root marginal given the extremal boundary condition $\SIGMA^+$ asymptotically coincides with the unconditional marginal.

The proof of \Prop~\ref{Lemma_condMarg} is delicate because the boundary condition $\SIGMA^+$ depends on the tree $\vT^{(2\ell)}$.
Indeed, it seems hopeless to confront these dependencies head on by first passing down the tree to construct $\SIGMA^+$ and to subsequently work up the tree to calculate marginals.
To sidestep this problem we devise a quantity that recovers the Markov property of the random tree.
Specifically, with each variable node $x\in\partial^{2k} o$, $k>0,$ of $\vT^{(2\ell)}$ we will associate a carefully defined quantity $\ETA_x^{(\ell)}\in\RR\cup\{\pm\infty\}$ that gauges how strongly $x$ can nudge its (grand-)parent variable $y$ towards the truth value mandated by $\SIGMA^+_y$.
This random variable $\ETA_x^{(\ell)}$ will turn out to be essentially independent of the top $2k$ levels  of the tree.
In effect, we will discover that the distribution of $\ETA_o^{(\ell)}$ can be approximated by the $k$-fold application of a suitable operator that will turn out to be a $W_1$-contraction.
Taking limits $k,\ell\to\infty$ carefully will then complete the proof.

To facilitate this construction we need to count satisfying assignments of sub-formulas of $\vT^{(2\ell)}$ subject to certain boundary conditions.
Specifically, for a variable $x$ we let $\vT_x^{(2\ell)}$ be the sub-formula of $\vT^{(2\ell)}$ comprising $x$ and its progeny.
Moreover, for a satisfying assignment $\tau\in S(\vT^{(2\ell)})$ we let
\begin{align*}
S(\vT_x^{(2\ell)},\tau)&=\cbc{\chi\in S(\vT_x^{(2\ell)}):\forall y\in V(\vT_x^{(2\ell)})\cap\partial^{2\ell}(\vT,o):\chi_y=\tau_y },&
Z(\vT_x^{(2\ell)},\tau)&=\abs{S(\vT_x^{(2\ell)},\tau)}.
\end{align*}
In words, $S(\vT_x^{(2\ell)},\tau)$ contains all satisfying assignments of $\vT_x^{(2\ell)}$ that comply with the boundary condition induced by $\tau$.
As a final twist, for $t=\pm1$ we also need the number
\begin{align*}
Z(\vT_x^{(2\ell)},\tau,t)&=\abs{\cbc{\chi\in S(\vT_x^{(2\ell)},\tau):\chi_x=t}}
\end{align*}
of satisfying assignments of $\vT_x^{(2\ell)}$ that agree with $\tau$ on the boundary and assign value $t$ to $x$.

The protagonist of the proof of \Prop~\ref{Lemma_condMarg} is the log-likelihood ratio
\begin{align}\label{eta}
\ETA_{x}^{(\ell)}&
= \log\frac{Z(\vT_{x}^{(2\ell)},\SIGMA^+,\SIGMA_x^+)}{Z(\vT_{x}^{(2\ell)},\SIGMA^+,-\SIGMA_x^+)}\in\RR\cup\{\pm\infty\}&&(x\in V(\vT^{(2\ell)})),
\end{align}
with the conventions $\log 0=-\infty$, $\log\infty=\infty$.
Thus, $\ETA_{x}^{(\ell)}$ gauges how likely a random satisfying assignment $\SIGMA$ of $\vT_x^{(2\ell)}$ subject to the $\SIGMA^+$-boundary condition is to set $x$ to its designated value $\SIGMA_x^+$.

To get a handle on the $\ETA_{x}^{(\ell)}$, we show that these quantities can be calculated by propagating the extremal boundary condition $\SIGMA^+$ bottom--up toward the root of the tree.
Specifically, we consider the operator 
\begin{align*}
\LDEP_{\vT^{(2\ell)}}&:(-\infty,\infty]^{V(\vT^{(2\ell)})}\to(-\infty,\infty]^{V(\vT^{(2\ell)})},
&\eta\mapsto\hat\eta
\end{align*}
defined as follows.
For all $x\in\partial^{2\ell}o$ we set $\hat\eta_x=\infty$.
Moreover, for a variable $x\in\partial^{2k}o$ with $k<\ell$ with children $a_1,\ldots,a_j$ and grandchildren $y_1\in\partial a_1\setminus\{x\},\ldots,y_j\in\partial a_j\setminus\{x\}$ we define
\begin{align}\label{eqhatetax}
\hat\eta_x&=-\sum_{i=1}^{j}{\SIGMA_x^+}\sign(x,a_i)\log\frac{1-{\SIGMA_x^+}\sign(x,a_i)\tanh(\eta_{y_i}/2)}{2}.
\end{align}
It may not be apparent that the above sum is well-defined as a $-\infty$ summand might occur.
However, the next lemma rules this out and shows that $\ell$-fold iteration of $\LDEP_{\vT^{(2\ell)}}$ from all-$+\infty$ yields $\ETA^{(\ell)}=(\ETA_x^{(\ell)})_{x\in V(\vT^{(2\ell)})}$.

\begin{lemma}\label{Lemma_upthetree}
The operator $\LDEP_{\vT^{(2\ell)}}$ is well-defined and 
$\mathrm{LL}^{+\, (\ell)}_{\vT^{(2\ell)}}(\infty,\ldots,\infty)= \ETA^{(\ell)}.$
\end{lemma}

\noindent
We defer the proof of \Lem~\ref{Lemma_upthetree} to \Sec~\ref{Sec_Lemma_upthetree}.

The next aim is to approximate the $\ell$-fold iteration of $\LDEP_{\vT^{(2\ell)}}$, and specifically the distribution of the value $\ETA_o^{(\ell)}$ associated with the root, via a non-random operator $\cP(\RR)\to\cP(\RR)$.
To this end we need to cope with the $\pm\infty$-entries of the vector $\ETA^{(\ell)}$, a task that we solve by  bounding $\ETA_x^{(\ell)}$ for variables $x$ near the top of the tree.

\begin{lemma}\label{Lemma_boundary}
There exist $c=c(d)>0$ and a sequence $(\eps_k)_{k\geq1}$ with $\lim_{k\to\infty}\eps_k=0$ such that  for any $k>0$, $\ell>ck$ we have $\pr[\max_{x\in \partial^{2k} o}|\ETA_{x}^{(\ell)}|\leq c k]>1-\eps_k$.
\end{lemma}

\noindent
The proof of \Lem~\ref{Lemma_boundary}, based on a percolation argument, can be found in \Sec~\ref{Sec_Lemma_boundary}.
We continue to denote by $c$ and $(\eps_k)_k$ the number and the sequence supplied by \Lem~\ref{Lemma_boundary}.

Guided by \Lem~\ref{Lemma_boundary} we consider the vector $\bar\ETA^{(\ell,k)}$ of truncated log-likelihood ratios
\begin{align*}
\bar\ETA^{(\ell,k)}_x&=\begin{cases}-ck&\mbox{ if $x\in\partial^{2k}o$ and }\ETA_{x}^{(\ell)}<-ck,\\ck&\mbox{ if $x\in\partial^{2k}o$ and }\ETA_{x}^{(\ell)}>ck,\\ \ETA_{x}^{(\ell)}&\mbox{ otherwise.}\end{cases}
\end{align*}
Further, let
$$\ETA^{(\ell,k)}=\LDEPk_{\vT^{(2\ell)}}(\bar\ETA^{(\ell,k)})$$
be the result of $k$ iterations of $\LDEP_{\vT^{(2\ell)}}(\nix)$ starting from $\bar\ETA^{(\ell,k)}$.

\begin{corollary}\label{Cor_boundary}
For any $\ell>ck$ we have $\dTV(\ETA^{(\ell,k)}_o,\ETA^{(\ell)}_o)<\eps_k.$
\end{corollary}
\begin{proof}
This follows from \Lem~\ref{Lemma_upthetree} and  \Lem~\ref{Lemma_boundary}, which shows that the truncation is inconsequential with probability at least $1-\eps_k$.
\end{proof}

We are ready to introduce the operator $\cP(\RR)\to\cP(\RR)$ that mimics $\LDEP_{\vT^{(2\ell)}}$.
Specifically, $\LDEP_d:\cP(\RR)\to\cP(\RR)$ maps $\rho\in\cP(\RR)$ to the distribution of
\begin{align}\label{eqNudgedFixedPoint}
-\sum_{i=1}^{\vd}\vs_i\log\frac{1-\vs_i\tanh(\ETA_{\rho,i}/2)}{2}.
\end{align}
We emphasise the subtle difference between \eqref{eqNudgedFixedPoint} and \eqref{Eq_BTreeOperator}, which involves two {independent} signs $\vs_i,\vs_i'$.
The next lemma establishes the connection between the random operator $\LDEP_{\vT^{(2\ell)}}$ and the operator $\LDEP_d$.
Namely, let $\rho^{(\ell,k)}$ be the distribution of $\ETA^{(\ell,k)}_o$.
Moreover, let $\bar\rho^{(\ell-k)}$ be the distribution of
$$\ETA^{(\ell-k)}_o\vecone\{-ck<\ETA^{(\ell-k)}_o<ck\} +ck\vecone\{ck<\ETA^{(\ell-k)}_o\}-ck\vecone\{\ETA^{(\ell-k)}_o<-ck\},$$
i.e.,  the truncation of $\ETA^{(\ell-k)}_o$.

\begin{lemma}\label{Lemma_cN}
For $\ell>ck$ we have $\rho^{(\ell,k)}=\LDEPk_d(\bar\rho^{(\ell-k)})$.
\end{lemma}

We prove \Lem~\ref{Lemma_cN} in \Sec~\ref{Sec_Lemma_cN}.
Recalling $\varphi$ from \eqref{eqll}, as in the proof of \Prop~\ref{Prop_BP} we let
 $\rho_d=\varphi(\pi_d)$  be the distribution of the log-likelihood ratio $\log(\MU_{\pi_d}/(1-\MU_{\pi_d}))$.

\begin{lemma}\label{Lemma_nudge}
The operator $\LDEP_d$ is a $W_1$-contraction with unique fixed point $\rho_d$.
\end{lemma}

\noindent
The proof of \Lem~\ref{Lemma_nudge} can be found in \Sec~\ref{Sec_Lemma_nudge}. Let $(\rho^{(\ell)})_\ell$ be the sequence of distributions of $(\vec\eta_o^{(\ell)})_\ell$.
As an immediate consequence we obtain the limit of the sequence $(\rho^{(\ell)})_\ell$.
We recall $\psi$ from \eqref{eqll}.

\begin{corollary}\label{Cor_cN}
The sequence $(\psi(\rho^{(\ell)}))_{\ell\geq0}$ converges weakly to $\pi_d$.
\end{corollary}
\begin{proof}
This follows from \Cor~\ref{Cor_boundary}, \Lem~\ref{Lemma_cN}, \Lem~\ref{Lemma_nudge} and the continuous mapping theorem.
\end{proof}

\begin{proof}[Proof of \Prop~\ref{Lemma_condMarg}]
Set
$\vec\vartheta_o^{(\ell)} = (\LDEPl_{\vT^{(2\ell)}}(0,\ldots,0))_o = \log(\mu_{\vT^{(2\ell)}}(\SIGMA_o=1)/\mu_{\vT^{(2\ell)}}(\SIGMA_o=-1)).$
Then 
\begin{align*}
\mu_{\vT^{(2\ell)}}(\SIGMA_o=1)&= \psi(\vec\vartheta_o^{(\ell)})&\mbox{and}&&
\mu_{\vT^{(2\ell)}}(\SIGMA_o=1\mid\SIGMA_{\partial^{2\ell}o}=\SIGMA^{+}_{\partial^{2\ell}o}) = \psi(\vec\eta_o^{(\ell)}).
\end{align*}
Moreover,  \Lem~\ref{Lemma_extremal} shows that $0 \leq \psi(\vec\vartheta_o^{(\ell)}) \leq  \psi(\vec\eta_o^{(\ell)}) \leq 1$.
Further, 
\Lem~\ref{Lemma_nudge} implies that $ \psi(\vec\vartheta_o^{(\ell)})$ converges weakly to $\pi_d$.  
Finally, \Cor~\ref{Cor_cN} implies that $\psi(\vec\eta_o^{(\ell)})$ also converges weakly to $\pi_d$, whence
\begin{align*}
\lim_{\ell\to\infty}\Erw{\abs{\psi(\vec\eta_o^{(\ell)}) - \psi(\vec\vartheta_o^{(\ell)})}} = \lim_{\ell\to\infty}\abs{\Erw[\psi(\vec\vartheta_o^{(\ell)})]-\Erw[\psi(\vec\eta_o^{(\ell)})]} = 0,
\end{align*}
which directly implies the assertion.
\end{proof}

\begin{proof}[Proof of \Prop~\ref{Prop_uniqueness}]
The proposition follows immediately from \Lem~\ref{Lemma_extremal} and \Prop~\ref{Lemma_condMarg}.
\end{proof}

\subsection{Proof of \Lem~\ref{Lemma_extremal}}\label{Sec_Lemma_extremal}
The proof is by induction on the height of the tree.
The following claim summarises the main step of the induction.

\begin{claim}\label{Lemma_Noela}
For all $\ell\geq0$, all variables $x$ of $\vT^{(2\ell)}$ and all satisfying assignments $\tau\in S(\vT^{(2\ell)})$ we have
\begin{align}\label{eqLemma_Noela1}
\frac{Z(\vT_x^{(2\ell)},\tau,\SIGMA_x^+)}{Z(\vT_x^{(2\ell)},\tau)}
&\leq\frac{Z(\vT_x^{(2\ell)},\SIGMA^+,\SIGMA_x^+)}{Z(\vT_x^{(2\ell)},\SIGMA^+)}.
\end{align}
\end{claim}
\begin{proof}
For boundary variables $x\in\partial^{2\ell} o$ there is nothing to show because the r.h.s.\ of \eqref{eqLemma_Noela1} equals one.
Hence, consider a variable $x\in\partial^{2k}o$ for some $k<\ell$.
If $Z(\vT_x^{(2\ell)},\tau,\SIGMA_x^+)=0$, then \eqref{eqLemma_Noela1} is trivially satisfied.
Hence, assume that $Z(\vT_x^{(2\ell)},\tau,\SIGMA_x^+)>0$.
Let $a_1^+,\ldots,a_g^+$ be the children (clauses) of $x$ with $\sign(x,a_i^+)=\SIGMA_x^+$.
Also let $y_1,\ldots,y_g$ be the children (variables) of $a_1^+,\ldots,a_g^+$.
Similarly, let $a_1^-,\ldots,a_h^-$ be the children of $x$ with $\sign(x,a_i^-)=-\SIGMA_x^+$ and let 
$z_1,\ldots,z_h$ be their children.
We claim that for all $\tau\in S(\vT^{(2\ell)})$,
\begin{align}\label{eqLemma_Noela2}
Z(\vT_x^{(2\ell)},\tau,\SIGMA_x^+)&=
	\prod_{i=1}^gZ(\vT_{y_i}^{(2\ell)},\tau)
	\prod_{i=1}^hZ(\vT_{z_i}^{(2\ell)},\tau,\SIGMA_{z_i}^+),&
Z(\vT_x^{(2\ell)},\tau,-\SIGMA_x^+)&=\prod_{i=1}^gZ(\vT_{y_i}^{(2\ell)},\tau,-\SIGMA_{y_i}^+)
		\prod_{i=1}^hZ(\vT_{z_i}^{(2\ell)},\tau).
\end{align}
For setting $x$ to $\SIGMA_x^+$ satisfies $a_1^+,\ldots,a_g^+$; hence, 
arbitrary satisfying assignments of the sub-trees $\vT_{y_i}^{(2\ell)}$ can be combined, which explains the first product.
By contrast, upon assigning $x$ the value $\SIGMA_x^+$ we need to assign the variables $z_i$ the values $\SIGMA_{z_i}^+$ so that they satisfy the clauses $a_i^-$.
This leaves us with  $Z(\vT_{z_i}^{(2\ell)},\tau,\SIGMA_{z_i}^+)$ possible satisfying assignments of the sub-trees $\vT_{z_i}^{(2\ell)}$; hence the second product, and we obtain the left equation.
A similar argument yields the right one.
Dividing the two expressions from \eqref{eqLemma_Noela2} and invoking the induction hypothesis (for $k+1$), we obtain
\begin{align*}
\frac{Z(\vT_x^{(2\ell)},\tau,-\SIGMA_x^+)}{Z(\vT_x^{(2\ell)},\tau,\SIGMA_x^+)}
	&=\prod_{i=1}^g\frac{Z(\vT_{y_i}^{(2\ell)},\tau,-\SIGMA_{y_i}^+)}
			{Z(\vT_{y_i}^{(2\ell)},\tau)}\cdot
			\prod_{i=1}^h\frac{Z(\vT_{z_i}^{(2\ell)},\tau)}
			{Z(\vT_{z_i}^{(2\ell)},\tau,\SIGMA_{z_i}^+)}\\
	&\geq
			\prod_{i=1}^g\frac{Z(\vT_{y_i}^{(2\ell)},\SIGMA^+,-\SIGMA_{y_i}^+)}
			{Z(\vT_{y_i}^{(2\ell)},\SIGMA^+)}\cdot
			\prod_{i=1}^h\frac{Z(\vT_{z_i}^{(2\ell)},\SIGMA^+)}
			{Z(\vT_{z_i}^{(2\ell)},\SIGMA^+,\SIGMA_{z_i}^+)}
		=\frac{Z(\vT_x^{(2\ell)},\SIGMA^+,-\SIGMA_x^+)}{Z(\vT_x^{(2\ell)},\SIGMA^+,\SIGMA_x^+)},
\end{align*}
completing the induction.
\end{proof}

\begin{proof}[Proof of \Lem~\ref{Lemma_extremal}]
The assertion follows by applying Claim~\ref{Lemma_Noela} to $x=o$.
\end{proof}

\subsection{Proof of \Lem~\ref{Lemma_upthetree}}\label{Sec_Lemma_upthetree}
To show that $\LDEP_{\vT^{(2\ell)}}$ is well defined we verify that, in the notation of \eqref{eqhatetax}, $\hat\eta_x\in(-\infty,\infty]$ for all $x$.
Indeed, in the expression on the r.h.s.\ of \eqref{eqhatetax} a $\pm\infty$ summand can arise only from variables $y_i$ with $\eta_{y_i}=\infty$.
But the definition of $\SIGMA^+$ ensures that such $y_i$ either render a zero summand if $\SIGMA_x^+\sign(x,a_i)=-1$, or a $+\infty$ summand if $\SIGMA_x^+\sign(x,a_i)=1$.
Thus, the sum is well-defined and $\hat\eta_x\in(-\infty,\infty]$.

Further, to verify the identity $\ETA^{(\ell)}=\LDEPl_{\vT^{(2\ell)}}(\infty,\ldots,\infty)$, consider a variable $x$ of $\vT^{(2\ell)}$.
Let $a_1^+,\ldots,a_g^+$ be its children with $\sign(a_i^+,x)=\SIGMA_x^+$, let $y_1,\ldots,y_g$ be their children, let $a_1^-,\ldots,a_h^-$ be the children of $x$ with $\sign(a_i^-,x)=-\SIGMA_x^+$ and let $z_1,\ldots,z_h$ be their children.
Then \eqref{eqll} {and \eqref{eqLemma_Noela2}} yield
\begin{align}\nonumber
\ETA_x^{(\ell)}&=-\sum_{i=1}^g\log\frac{Z(\vT_{y_i}^{(2\ell)},\SIGMA^+,-\SIGMA_{y_i}^+)}{Z(\vT_{y_i}^{(2\ell)},\SIGMA^+)}
		+\sum_{i=1}^h\log\frac{Z(\vT_{z_i}^{(2\ell)},\SIGMA^+,\SIGMA_{z_i}^+)}{Z(\vT_{z_i}^{(2\ell)},\SIGMA^+)}
=-\sum_{i=1}^g\log\frac{1-\tanh(\ETA_{y_i}^{(\ell)}/2)}2+\sum_{i=1}^h\log\frac{1+\tanh(\ETA_{z_i}^{(\ell)}/2)}2.
\end{align}
The assertion follows because $\sign(x,a_i^+)\SIGMA_x^+=1$ and $\sign(x,a_i^-)\SIGMA_x^+=-1$.

\subsection{Proof of Lemma \ref{Lemma_boundary}}\label{Sec_Lemma_boundary}
The goal is to prove that for variables some distance away from level $2\ell$ of $\vT^{(2\ell)}$ the counts  $Z(\vT_{x}^{(2\ell)},\SIGMA^+,\pm1)$ are roughly of the same order of magnitude.
Approaching this task somewhat indirectly, we begin by tracing the logical implications of imposing a specific value $s=\pm1$ on a variable $x$ of the (possibly infinite) tree $\vT$.
Clearly, upon setting $x$ to the value $s$ a child (clause) $a$ of $x$ will be satisfied iff $x$ appears in $a$ with sign $s$.
In effect, all clauses $a$ with $\sign(a,x)\neq s$ need to be satisfied by their second variable $y$, a grandchild of $x$.
Thus, we impose the value $\sign(a,y)$ on $y$ and recurse down the tree.
Let $\vT_{x,s}$ denote the sub-tree of $\vT$ comprising $x$ and all other variables on which this process imposes specific values as well as all clauses that contain two such variables.
Clearly, for every leaf $y$ of $\vT_{x,s}$ the values imposed on $y$ happens to satisfy all child clauses of $y$ in $\vT$.
Let $\vN_{x,s}\in[1,\infty]$ be the number of variables in $\vT_{x,s}$.
The next lemma shows that the impact of a boundary condition on the marginal of $x$ can be bounded in terms of $\vN_{x,s}$.

\begin{claim} \label{deletetree}
Let $s\in\{\pm1\}$.
If $x\in\partial^{2k} o$ satisfies $\vN_{x,s}<\ell-k$ then  $Z(\vT_x^{(2\ell)},\tau)\leq 2^{\vN_{x,s}}Z(\vT_x^{(2\ell)},\tau,s)$.
\end{claim}
\begin{proof}
The construction of the implication tree $\vT_{x,s}$ imposes a truth value $\sigma_y$ on each variable $y$ of the tree that $y$ must inevitably take if $x$ gets assigned $s$.
Thus, $\vT_{x,s}$ comes with a satisfying assignment $\sigma\in S(\vT_{x,s})$ with $\sigma_x=s$.
For any leaf $y$ of $\vT_{x,s}$ every child clause $a$ of $y$ in the super-tree $\vT$ will be automatically satisfied by setting $y$ to $\sigma_y$ (because otherwise $a$ would have been included in $\vT_{x,s}$).
Hence, all the clauses of $\vT$ that are children of the leaves of $\vT_{x,s}$ are satisfied by $\sigma$.
Moreover, because $\vN_{x,s}<\ell-k$, any leaf $y$ of $\vT_{x,s}$ has distance less than $2\ell$ from $o$.
Thus, the assignment $\sigma$ does not clash with the boundary condition $\tau$.
As a consequence, for any {$\chi\in S(\vT_x^{(2\ell)},\tau)$}  we obtain another satisfying assignment 
$\chi'\in S(\vT_x^{(2\ell)},\tau)$ by letting
\begin{align*}
\chi'_z&=\begin{cases}
\sigma_z&\mbox{ if }z\in V(\vT_{x,s}),\\
\chi_z&\mbox{ otherwise}.
\end{cases}
\end{align*}
Moreover,  under the map $\chi\mapsto\chi'$ the number of inverse images of any assignment $\chi'$ is bounded by the total number $2^{\vN_{x,s}}$ of different truth assignments of the variables $V(\vT_{x,s})$.
Therefore, $Z(\vT_x^{(2\ell)},\tau)\leq 2^{\vN_{x, s}}Z(\vT_x^{(2\ell)},\tau,s)$.
\end{proof}

As a next step we bound the random variable $\vN_{x,s}$.

\begin{claim} \label{nodesdown}
There exists a number $\alpha=\alpha(d)>0$ such that 
$\pr\brk{\vN_{o,s}\geq u}\leq \exp\bc{-u\alpha}/\alpha$ for all $u\geq0$, $s\in\{\pm1\}$.
\end{claim}
\begin{proof}
In the construction of $\vT_{o,s}$ we only propagate along clauses in which the parent variable is forced to take a value that fails to satisfy the clause.
Since the signs are uniformly random, the number of such child clauses has distribution $\Po(d/2)$.
Therefore, $\vN_{o,s}$ is bounded by the total progeny of a Galton-Watson process with  $\Po(d/2)$ offspring.
The assertion therefore follows from the tail bound for such processes (e.g.,~\cite[eq.~(11.7)]{AS}).
\end{proof}

As a final preparation toward the proof of Lemma \ref{Lemma_boundary} we need a bound on the size of the $2k$-th level of $\vT$.

\begin{claim}\label{Lemma_GW}
We have $\lim_{k\to\infty}\pr\brk{|\partial^{2k}o|>2d^k+k}=0$.
\end{claim}
\begin{proof}
Since every clause of $\vT$ has precisely one child, the size of level $2k$ of $\vT$ coincides with the size of the $k$-th level of a $\Po(d)$ Galton-Watson tree.
Therefore, the assertion follows from standard tail bounds for Galton-Watson processes (e.g.,~\cite[eq.~(11.7)]{AS}).
\end{proof}

\begin{proof}[Proof of Lemma \ref{Lemma_boundary}]
Claim~\ref{nodesdown} ensures that for a large enough constant $c=c(d)>0$ and all large enough $k$,
\begin{align}  \label{Eq_TailBoundBoundary}
\Pr \bc{\vN_{o,\pm1} \geq ck}\leq (2d)^{-k}.
\end{align}   
Combining \eqref{Eq_TailBoundBoundary} with Claim~\ref{Lemma_GW} and using the union bound, we obtain a sequence $\eps_k\to0$ such that
\begin{align}  \label{Eq_TailBoundBoundary2}
\Pr \bc{\forall x\in\partial^{2k}o:\vN_{x,\pm1} < ck}\geq 1-\eps_k.
\end{align}   
Further, if $x\in\partial^{2k}o$ satisfies $\vN_{x,\pm1} < ck$ and $\ell>(1+c)k$,  Claim~\ref{deletetree} ensures that for all $x\in\partial^{2k}o$,
\begin{equation}\begin{split}
\abs{\vec\eta_{x}^{(\ell)}}&\leq \log \frac{Z(\vT_x^{(2\ell)},\SIGMA^+)}{Z(\vT_x^{(2\ell)},\SIGMA^+,1)}
	+\log \frac{Z(\vT_x^{(2\ell)},\SIGMA^+)}{Z(\vT_x^{(2\ell)},\SIGMA^+,-1)} \leq 
	 \vN_{x,1}+\vN_{x,-1} < 2ck. \label{Eq_BoundaryZ2}
\end{split}
\end{equation}
Combining \eqref{Eq_TailBoundBoundary2} and \eqref{Eq_BoundaryZ2} completes the proof.
\end{proof}

\subsection{Proof of \Lem~\ref{Lemma_cN}}\label{Sec_Lemma_cN}
A straightforward induction shows that for any $p\in\cP(\RR)$ the result $p^{(k)}=\LDEPk_d(p)$ of the $k$-fold application of $\LDEP_d$ coincides with the distribution of the root value of the random operator $\LDEPk_{\vT^{(2k)}}$ applied to a vector $(\ETA_x)_{x\in V(\vT^{(2k)})}$ of independent samples from $p$.
Indeed, for $k=1$ the claim is immediate from the definitions.
Moreover, for the inductive step we notice that the $k$-fold application of $\LDEP_d$ comes down to applying $\LDEP_d$ once to the outcome of the $(k-1)$-fold application.
By the induction hypothesis, $$p^{(k-1)} =\bc{\mathrm{LL}^{+\,(k-1)}_{\vT^{(2(k-1))}}(\ETA_x)_x}_o.$$
Finally, applying $\LDEP_d$ to $p^{(k-1)}$ implies the assertion because the first layer of $\vT^{(2k)}$ is independent of the subtrees rooted at the grandchildren $\partial^2o$ of the root, which are distributed as independent random trees $\vT^{(2(k-1))}$. 
The lemma follows from applying this identity to $p=\bar{\rho}^{(\ell-k)}$.

\subsection{Proof of \Lem~\ref{Lemma_nudge}}\label{Sec_Lemma_nudge}
The operator $\LDEP_d$ maps the space  $\cW_1(\RR)$ into itself because the derivative of $x\mapsto\log((1-\tanh(x/2))/2)$ is bounded by one in absolute value for all $x\in\RR$.
We proceed to show that $\LDEP_d:\cW_1(\RR)\to\cW_1(\RR)$ is a contraction.
Thus, consider a sequence of independent random pairs $(\ETA_i,\ETA_i')_{i\geq1}$ with $\ETA_i\disteq\rho$, $\ETA_i'\disteq\rho'$.
Then
\begin{align*}
W_1(\LDEP_d(\rho),\LDEP_d(\rho'))&\leq
\Erw\abs{\sum_{i=1}^{\vd}\vs_i\log\frac{1-\vs_i\tanh(\ETA_i/2)}{1-\vs_i\tanh(\ETA_i'/2)}}
\leq d\Erw\abs{\log\frac{1-\vs_1\tanh(\ETA_1/2)}{1-\vs_1\tanh(\ETA_1'/2)}}.
\end{align*}
Since the function $z\mapsto\log(1+\tanh(z/2))$ is monotonically increasing, we obtain
\begin{align*}
\abs{\log\frac{1+\tanh(\ETA_1/2)}{1+\tanh(\ETA_1'/2)}}&=
\abs{\int_{\ETA_1'}^{\ETA_1}\frac{\partial\log(1+\tanh(z/2))}{\partial z}\dd z}
=\int_{\ETA_1\wedge \ETA_1'}^{\ETA_1\vee\ETA_1'}\frac{1-\tanh(z/2)}{2}\dd z,\\
\abs{\log\frac{1-\tanh(\ETA_1/2)}{1-\tanh(\ETA_1'/2)}}&=
\abs{\int_{\ETA_1'}^{\ETA_1}\frac{\partial\log(1-\tanh(z/2))}{\partial z}\dd z}
=\int_{\ETA_1\wedge \ETA_1'}^{\ETA_1\vee\ETA_1'}\frac{1+\tanh(z/2)}{2}\dd z.
\end{align*}
Hence,
$W_1(\LDEP_d(\rho),\LDEP_d(\rho'))\leq d\Erw\abs{\ETA_1-\ETA_1'}/2$ 
and therefore
$W_1(\LDEP_d(\rho),\LDEP_d(\rho'))\leq dW_1(\rho,\rho')/2$.

Finally, we observe that $\rho_d$ is a fixed point of $\LDEP_d$.
Indeed, \Prop~\ref{Prop_BP} implies that $\ETA^{\rho_d}$ and $-\ETA^{\rho_d}$ are identically distributed.
Therefore, if $\vs_i,\vs_i'\in\{\pm1\}$ are uniform and independent, we obtain
$$\vs_i\log\bc{\bc{1-\vs_i\tanh(\ETA_{\rho_d,i}/2)}/2}\disteq \vs_i\log\bc{\bc{1+\vs_i'\tanh(\ETA_{\rho_d,i}/2)}/2}.$$
Hence, recalling the definitions \eqref{Eq_BTreeOperator} and \eqref{eqNudgedFixedPoint} of the operators, we see that
$\LDEP_d(\rho_d)=\LDE_d(\rho_d)=\rho_d$.

\subsection{Proof of \Thm~\ref{Thm_BP}}
Consider the sub-formula  $\nabla^{2\ell}(\PHI,x_1)$ of $\PHI$ obtained by deleting all clauses and variables at distance greater than $2\ell$ from $x_1$.
By design, we can couple $\nabla^{2\ell}(\PHI,x_1)$ and $\vT^{(2\ell)}$ such that both coincide \whp{}
Therefore, since any satisfying assignment of $\PHI$ induces a satisfying assignment of $\vT^{(2\ell)}$, \Prop~\ref{Prop_uniqueness} implies the Gibbs uniqueness property~\eqref{eqGU}.
Furthermore, because \Prop~\ref{Fact_BP} shows that Belief Propagation correctly computes the root marginal 
$\mu_{\vT^{(2\ell)}}(\SIGMA_o=1)$, \eqref{eqThm_BP} follows from \eqref{eqGU}.

\subsection{Proof of \Cor~\ref{Cor_BP}}\label{Sec_Cor_BP}
Let $\pi_d^{(\ell)}=\BP^{(\ell)}(\delta_{1/2})$.
Thanks to \Prop~\ref{Prop_BP} it suffices to prove that 
\begin{align}\label{eqCor_BP_1}
\lim_{\ell\to\infty}\limsup_{n\to\infty}\Erw[W_1(\pi_{\PHI},\pi_d^{(\ell)})]=0.
\end{align}
Hence, fix $\eps>0$, pick a large $\ell=\ell(\eps)>0$ and a larger $L=L(\ell)>0$.
A routine second moment calculation shows that for any possible outcome $T$ of $\vT^{(2\ell)}$ the number $X_T$ of variables $x_i$ of $\PHI$ such that $\nabla^{2\ell}(\PHI,x_i)=T$ satisfies $X_T=n\pr\brk{\vT^{(2\ell)}=T}+o(n)$ \whp{}
Hence, \whp\ $\PHI$ admits a coupling $\gamma_{\PHI}$ of $\vT^{(2\ell)}$ and a uniform variable $\vec i$ on $[n]$ such that 
$\gamma(\{\nabla^{2\ell}(\PHI,x_{\vec i})=\vT^{(2\ell)}\})=1-o(1)$.
Further, \Thm~\ref{Thm_BP} implies that given $\nabla^{2\ell}(\PHI,x_{\vec i})=\vT^{(2\ell)}$ we have
\begin{align}\label{eqCor_BP_2}
\pr\brk{\abs{\mu_{\PHI}(\SIGMA_{x_i}=1)-\mu_{\vT^{(2\ell)}}(\TAU_o=1)}>\eps}<\eps,
\end{align}
provided $\ell$ is large enough.
Finally, \Lem~\ref{Fact_BP} implies together with a straightforward induction on $\ell$ that $\pi_d^{(\ell)}$ is the distribution of $\mu_{\vT^{(2\ell)}}(\TAU_o=1)$.
Therefore, \eqref{eqCor_BP_1} follows from \eqref{eqCor_BP_2}.

\subsection{Proof of \Cor~\ref{Cor_uniqueness}}\label{Sec_Cor_uniqueness}
Fix $\eps>0$ and pick a small $\xi=\xi(\eps)>0$ and large $\ell=\ell(\xi)>0$.
Since $k$ is fixed independently of $n$, \Thm~\ref{Thm_BP} shows that \whp{}
\begin{align}\label{eqCor_uniqueness_1}
\sum_{i=1}^k\max_{\tau\in S(\PHI)}
		\abs{\mu_{\PHI}(\SIGMA_{x_i}=1\mid\SIGMA_{\partial^{2\ell}x_i}=\tau_{\partial^{2\ell}x_i})
			-\mu^{(\ell)}_{\PHI,x_i}(1)} &<\xi.
\end{align}
Further, the smallest pairwise distance between $x_1,\ldots,x_n$ exceeds $4\ell$ \whp{}
Therefore, we can draw a sample $\SIGMA$ from $\mu_{\PHI}$ in two steps.
First, draw $\SIGMA'$ from $\mu_{\PHI}$.
Then, independently re-sample assignments of all the variables in $\nabla^{2\ell-2}(\PHI,x_i)$ from $\mu_{\PHI}(\nix|\SIGMA'_{\partial^{2\ell}x_i})$  for $i=1,\ldots,k$.
The resulting assignment $\SIGMA''$ has distribution $\mu_{\PHI}$ and the values $\SIGMA''_{x_i}$, $i\in[k]$, are mutually independent given $\SIGMA'$.
Finally, since \eqref{eqCor_uniqueness_1} shows that conditioning on the boundary conditions $\SIGMA'_{\partial^{2\ell}x_i}$ is inconsequential \whp, we obtain the assertion by taking $\eps\to0$ sufficiently slowly.

\section{Proof of \Prop~\ref{Prop_Bethe}}\label{Sec_Prop_Bethe}

\subsection{Outline} \label{Sec_Proof_Prop_Bethe}
The proof is based on a natural coupling of the random formulas $\PHI_n$ and $\PHI_{n+1}$ with $n$ and $n+1$ variables, respectively.
Specifically, let 
\begin{align}\label{eqm''}
\vm'&\disteq\Po(dn/2-d/2),&\DELTA''&\disteq\Po(d/2),&\DELTA'''&\disteq\Po(d)
\end{align}
be independent random variables.
Moreover, let $\PHI'$ be a random formula with $n$ variables and $\vm'$ clauses, chosen independently and uniformly from the set of all $4n(n-1)$ possible clauses.
Then obtain $\PHI''$ from $\PHI'$ by adding  another $\DELTA''$ uniformly random and independent clauses.
Moreover, obtain $\PHI'''$ from $\PHI'$ by adding one variable $x_{n+1}$ along with $\DELTA'''$ clauses, chosen uniformly and independently from the set of all $8n$ possible clauses that contain $x_{n+1}$ and another variable from the set $\{x_1,\ldots,x_n\}$.

\begin{fact}\label{Fact_ASS}
We have $\PHI''\disteq\PHI_n$ and $\PHI'''\disteq\PHI_{n+1}$; therefore,
\begin{align}\label{eqFact_ASS}
\Erw[\log(Z(\PHI_{n+1})\vee1)]-\Erw[\log(Z(\PHI_{n})\vee1)]=\Erw\brk{\log\frac{Z(\PHI''')\vee1}{Z(\PHI')\vee1}}-\Erw\brk{\log\frac{Z(\PHI'')\vee1}{Z(\PHI')\vee1}}.
\end{align}
\end{fact}

\noindent
Hence, the proof of \Prop~\ref{Prop_Bethe} boils down to establishing the following two statements.

\begin{proposition}\label{Lemma_PHI''}
We have
$\displaystyle 
\lim_{n\to\infty}\Erw\log\frac{Z(\PHI'')\vee1}{Z(\PHI')\vee1}=
\frac d2\Erw\brk{\log\bc{1-\MU_{\pi_d,1}\MU_{\pi_d,2}}}.$
\end{proposition}

\begin{proposition}\label{Lemma_PHI'''}
We have
$\displaystyle \lim_{n\to\infty}\Erw\log\frac{Z(\PHI''')\vee1}{Z(\PHI')\vee1}=
\Erw\brk{\log\bc{\sum_{\sigma\in\cbc{\pm1}}\prod_{i=1}^{\vec d}
\bc{1-\vecone\cbc{\sigma\neq\vs_i}\MU_{\pi_d,i}}}}.	$
\end{proposition}

\noindent
Further, to prove \Prop s~\ref{Lemma_PHI''} and \ref{Lemma_PHI'''} we `just' need to understand the impact of a bounded expected number of `local' changes (such as adding a random clause) on the partition function.

The proof strategy sketched in the previous paragraph is known as the Aizenman-Sims-Starr scheme.
The technique was originally deployed to study the Sherrington-Kirkpatrick spin glass model~\cite{Aizenman}, but has since found various applications to models on sparse random graphs (e.g., \cite{CKPZ,Panchenko}).
By comparison to prior applications, the difficulty here is that we apply this technique to a model with hard constraints.
In effect, while typically the addition of a single clause will only reduce the number of satisfying assignments by a bounded factor, occasionally a much larger change might ensue.
For instance, for any $0<d<2$ there is a small but non-zero probability that a single additional clause might close a `bicycle', i.e., a sequence of clauses that induce an implication chain $x_i\to \cdots\to \neg x_i\to\cdots \to x_i$.
Thus, a single unlucky clause might wipe out all satisfying assignments.

Suppose we wish to roughly estimate the change in the number of satisfying assignments upon going from $\PHI'$ to $\PHI'''$.
Clearly $Z(\PHI''')\leq2 Z(\PHI')$ because we only add one new variable.
But of course  $Z(\PHI''')$ might be much smaller than $Z(\PHI')$.
To obtain a bound, consider the new clauses $b_1,\ldots,b_{\DELTA'''}$ that were added along with $x_{n+1}$ and let $y_1,\ldots,y_{\DELTA'''}$ be the variables of $\PHI'$ where the new clauses attach.
Define an assignment $\chi:Y=\{y_1,\ldots,y_{\DELTA'''}\}\to\{\pm1\}$ by letting $\chi_{y_i}=\sign(y_i,b_i)$; thus, $\chi$ satisfies the $b_i$.
Further, let
\begin{align*}
S(\PHI',\chi)&=\cbc{\sigma\in S(\PHI'):\forall y\in Y:\sigma_{y}=\chi_y},&Z(\PHI',\chi)=|S(\PHI',\chi)|
\end{align*}
be the set and the number of satisfying assignments of $\PHI'$ that coincide with $\chi$ on $Y$.
Because each $\sigma\in S(\PHI',\chi)$ already satisfies all the new clauses regardless of the value assigned to $x_{n+1}$, we obtain $Z(\PHI''')\geq 2Z(\PHI',\chi)$.
Hence, it seems that we just need to lower bound $Z(\PHI',\chi)$.

To this end we could employ a process similar to the one that we applied in \Sec~\ref{Sec_Lemma_boundary}  to the tree $\vT$.
Generally, let $Y\subset\{x_1,\ldots,x_n\}$ be a set of variables and let $\chi\in\{\pm1\}^Y$ be an assignment.
The following process, known as the Unit Clause Propagation algorithm \cite{Dowling84}, chases the implications of imposing the assignment $\chi$ on $Y$:
\begin{quote}
while $\PHI'$ possesses a clause $a$ that has exactly one neighbouring variable $z\in\partial a$ on which the value $-\sign(z,a)$ has been imposed, impose the value $\sign(a,z')$ on the second variable $z'\in\partial a\setminus\cbc z$ of $a$.
\end{quote}
Let $\cI_\chi$ be the set of variables on which the process has imposed a value upon termination (including the initial set $Y$).
Unfortunately, it is possible that $\PHI'$ contains a clause $a$ on whose both variables $z,z'$ the `wrong' values $-\sign(a,z),-\sign(a,z')$ got imposed.
In other words, Unit Clause might be left with contradictions.
If such a clause exists we let $\vI_\chi=n$. Otherwise we set $\vI_\chi=|\cI_\chi|$.
We obtain the following lower bound on $Z(\PHI',\chi)$.

\begin{fact}\label{Fact_UC}
We have $Z(\PHI')\leq 2^{\vI_\chi}(Z(\PHI',\chi)\vee1)$.
\end{fact}
\begin{proof}
The inequality is trivially satisfied if $Z(\PHI')=0$ or $\vI_\chi=n$.
Hence, we may assume that $Z(\PHI')>0$ and that Unit Clause did not run into a contradiction.
Consequently, Unit Clause produced an assignment $\chi^*$ of the variables $\cI_\chi$ that satisfies all clauses $a$ of $\PHI'$ with $\partial a\cap\cI_\chi\neq\emptyset$.
Hence, for any satisfying assignment $\sigma\in S(\PHI')$ we obtain another satisfying assignment $\hat\sigma\in S(\PHI',\chi)$ by letting
$\hat\sigma=\chi^*_x\vecone\{x\in \cI_\chi\}+\sigma_x\vecone\{x\not\in \cI_\chi\}$, i.e., we overwrite the variables in $\cI_\chi$ according to $\chi^*$.
Clearly, under the map $\sigma\mapsto\hat\sigma$ an assignment $\hat\sigma\in S(\PHI',\chi)$ has at most $2^{\vI_\chi}$ inverse images.
\end{proof}

\noindent
Hence, we need an upper bound on $\vI_\chi$, which will be proven at the end of Section \ref{SubSec_UCP}.

\begin{lemma}\label{Lemma_wild}
There exists $C=C(d)>0$ such that for every set $Y\subset\{x_1,\ldots,x_n\}$ of size $|Y|\leq\log^2n$ and any $\chi\in\{\pm1\}^Y$ we have
$\Erw[\vI_\chi]\leq C|Y|^2.$
\end{lemma}
\noindent
Unfortunately, this first moment bound does not quite suffice for our purposes.
Indeed, \Lem~\ref{Lemma_wild} allows for the possibility that $\vI_\chi=n$ with probability $\Omega(1/n)$.
In combination with Fact~\ref{Fact_UC} this rough bound would lead to error terms that eclipse the `main' terms displayed in \Prop s~\ref{Lemma_PHI''} and~\ref{Lemma_PHI'''}.
But we cannot hope for a much better bound on $\vI_\chi$.
Indeed, $\pr\brk{\vI_\chi=n}=\Omega(1/n)$ because the graph $G(\PHI')$ likely contains a few short cycles and if $Y$ contains a variable on a short cycle, then there is a $\Omega(1)$ probability that Unit Clause will cause a contradiction.

Hence, we need to be more circumspect.
Previously we aimed for an assignment $\chi$ that satisfied {\em all} the new clauses $b_1,\ldots,b_{\DELTA'''}$ added upon going to $\PHI'''$.
But we still have the new variable $x_{n+1}$ at our disposal to at least satisfy a single clause $b_i$.
Hence, we can afford to start Unit Clause from an assignment $\chi'$ that differs from $\chi$ on a single variable.
Thus, for a set $Y$ of variables and $\chi\in\{\pm1\}^Y$ we define
\begin{align}\label{eqAchi}
\vA_\chi&=\min\cbc{\vI_{\chi'}:\chi'\in\{\pm1\}^Y,\ \sum_{y\in Y}\vecone\{\chi_y\neq\chi_y'\}\leq1}.
\end{align}

\begin{lemma}\label{Lemma_tame}
There exists $C'=C'(d)>0$ such that for every set $Y\subset\{x_1,\ldots,x_n\}$ of size $|Y|\leq\log^2n$ and any $\chi\in\{\pm1\}^Y$ we have
$\Erw[\vA_\chi^2]\leq C'|Y|^4.$
\end{lemma}

This second moment bound significantly improves over \Lem~\ref{Lemma_wild}.
For instance, \Lem~\ref{Lemma_tame} implies that the probability of an enormous drop $Z(\PHI''')\leq\exp(-\Omega(n))Z(\PHI')$ is bounded by $O(n^{-2})$.
Once more this estimate is about tight because there is an $\Omega(n^{-2})$ probability that a single new clause closes a bicycle.
As we shall see, with a bit of care the bound from \Lem~\ref{Lemma_tame} suffices to prove \Prop s~\ref{Lemma_PHI''} and~\ref{Lemma_PHI'''}.
Yet \Lem~\ref{Lemma_wild} has its uses, too, as it implies the following vital tail bound.

\begin{corollary}\label{Lemma_emp}
We have $\limsup_{n\to\infty}\displaystyle\Erw \brk{n\wedge\abs{\log\frac{\mu_{\PHI'}(\SIGMA_{x_1}=1)}{\mu_{\PHI'}(\SIGMA_{x_1}=-1)}}\mid Z(\PHI')>0}<\infty.$
\end{corollary}

We proceed to study Unit Clause Propagation in order to prove \Lem s~\ref{Lemma_wild}, \ref{Lemma_tame} and \Cor~\ref{Lemma_emp}.
Then we will prove \Prop s~\ref{Lemma_PHI''} and \ref{Lemma_PHI'''}, which imply \Prop~\ref{Prop_Bethe}.

\subsection{Unit Clause Propagation}
\label{SubSec_UCP}
\label{Proof_Lemma_Tame}
To avoid dependencies we  consider a binomial model $\PHI^\dagger$ of a random $2$-SAT formula with
variables $x_1,\ldots,x_n$, where each of the $4\binom n2$ possible (unordered) $2$-clauses is present with probability 
\begin{align}\label{eqp}
p=d/(4n)+n^{-4/3}
\end{align}
independently.
We define a random variable $\vA_\chi^\dagger$ on $\PHI^\dagger$ in perfect analogy to $\vA_\chi$.
Since the choice \eqref{eqp} of $p$ ensures that $\PHI^\dagger$ and $\PHI'$ can be coupled so that the former has more clauses than the latter with probability $1-o(n^{-2})$, it suffices to analyse $\vA_\chi^\dagger$.
Moreover, thanks to symmetry it suffices  to prove \Lem s~\ref{Lemma_wild} and~\ref{Lemma_tame}  under the assumption that the initial  set of variables is $Y=\{x_1,\ldots,x_\ell\}$, $\ell\leq\log^2n$.

At first glance investigating $\vA_\chi^\dagger$ appears to be complicated by the fact that \eqref{eqAchi} takes the minimum over all possible $\chi'$.
To sidestep this issue we will investigate a `comprehensive' propagation process whose progeny encompasses all the unit clauses that may result from any  $\chi'$.
In its first round this process pursues for each variable $x_i$, $i\leq\ell$, the Unit Clauses created by imposing either of the two possible truth values on $x_i$.
The effect will be the imposition of truth values on all variables at distance two from $Y$.
Subsequently we trace Unit Clause Propagation from the values imposed on the variables in $\partial^2Y$.
Hence, the difficulty of considering all $\chi'$ as in \eqref{eqAchi} disappears because the first step disregards $\chi$.

To deal with possible contradictions the process will actually operate on literals rather than variables.
Throughout each literal will belong to one of three possible categories: unexplored, explored, or finished.
Initially the $2\ell$ literals $x_1,\neg x_1,\ldots,x_\ell,\neg x_\ell$ qualify as explored and all others as unexplored.
Formally, we let $\cE_0=\{x_1,\neg x_1,\ldots,x_\ell,\neg x_\ell\}$, $\cU_0=\cbc{x_{\ell+1},\neg x_{\ell+1},\ldots,x_n,\neg x_n}$ and $\cF_0=\emptyset$.
Further, for $t\geq0$ we construct $\cE_{t+1}, \cU_{t+1}, \cF_{t+1}$ as follows.
If $\cE_t=\emptyset$, the process has terminated and we set $\cE_{t+1}=\cE_t,\cU_{t+1}=\cU_t,\cF_{t+1}=\cF_t.$
Otherwise, pick a literal $l_{t+1}\in\cE_t$ and let
$\cE_{t+1}'$ be the set of all literals $l'\in\cU_t$ such that $\PHI^\dagger$ features the clause 
$\neg l_{t+1}\vee l'$.
Further, let
\begin{align*}
\cU_{t+1}&=\cU_t\setminus\cE_{t+1}',&\cE_{t+1}&=(\cE_t\cup\cE_{t+1}')\setminus\{l_{t+1}\},&\cF_{t+1}&=\cF_t\cup\cbc{l_{t+1}}.
\end{align*}
Finally, the set $\cF_\infty=\bigcup_{t\geq1}\cF_t$ contains all literals upon which Unit Clause could impose the value `true' from any initial assignment $\chi$.
A contradiction might result only if $x_i,\neg x_i\in\cF_\infty$ for some $i>\ell$.

\begin{claim}\label{Lemma_size_of_implication_graphs}
For all $T > 8\ell/(2-d)$  we have 
$\Pr \brk{|\cF_\infty| > T} \leq \exp(- dT/36)$.
\end{claim}
\begin{proof}
Let $t \geq 0$.
Given $|\cU_t|$ and $|\cE_t|$ we have
\begin{align*}
\vX_{t+1} = \abs{\cE_{t+1}}-\abs{\cE_{t}}+\vecone\cbc{\cE_{t}\neq\emptyset} & \stackrel{d}{=} \Bin\bc{|\cU_t|\vecone\cbc{|\cE_{t}|\geq 0},p }.
\end{align*} 
Moreover, given $|\cU_t|$ and $|\cE_t|$ let $\vY_{t+1} \stackrel{d}{=}  \Bin\bc{2n - |\cU_t|\vecone\cbc{|\cE_{t}|\geq 0},p }$ be independent of $\vX_{t+1}$ and everything else, and set $\vX^{\geq}_{t+1} = \vX_{t+1} + \vY_{t+1}$.
Then $(\vX^{\geq}_t)_{t \geq 1}$ is an i.i.d.\ sequence of $\Bin(2n,p)$ random variables such that  $\vX^{\geq}_t \geq \vX_t$ for all $t$.
Hence, for any $T \geq 1$,
\begin{align}\label{eqLemma_size_of_implication_graphs3}
\pr\brk{|\cF_\infty|>T}=\pr\brk{|\cE_T| > 0}&\leq\pr\brk{\sum_{t=1}^T \vX_t >T-2\ell}\leq \pr\brk{\sum_{t=1}^T \vX_t^\geq  >T-2\ell} = \pr\brk{\Bin(2nT,p)>T-2\ell}.
\end{align}
Further, the Chernoff bound shows that for $T>8\ell/(2-d)$ (and $n$ large enough), 
\begin{align}\label{eqLemma_size_of_implication_graphs4}
\pr\brk{\Bin(2Tdn,p)>T-2\ell}&\leq
\exp\bc{-\min\cbc{(d-n^{-4/3}), (d-n^{-4/3})^2}\frac{2nTp}{3}}\leq\exp\bc{- \frac{dT}{36}},
\end{align}
Combining \eqref{eqLemma_size_of_implication_graphs3} and \eqref{eqLemma_size_of_implication_graphs4} completes the proof.
\end{proof}

Let $\PHI^*$ be the sub-formula of $\PHI^\dagger$ comprising all variables $x$ such that $x\in\cF_\infty$ or $\neg x\in\cF_\infty$ along with all clauses $a$ that contain two such variables.
Let $\vn^*$ be the number of variables of $\PHI^*$ and let $\vm^*$ be the number of clauses.

\begin{claim}\label{Claim_bicycle}
We have $\pr\brk{\vm^*\geq\vn^*-\ell+1}\leq O(\ell^2/n)$ and $\pr\brk{\vm^*>\vn^*-\ell+1}\leq O(\ell^4/n^2)$.
\end{claim}
\begin{proof}
We set up a graph representing the literals involved in the exploration process and the clauses that contain such literals.
Specifically, let $\neg\cF_\infty=\{\neg l:l\in\cF_\infty\}$ contain all negations of literals in $\cF_\infty$.
Moreover, let $\cG$ be the graph whose vertices are the literals $\cF_\infty\cup\neg\cF_\infty$ as well as all clauses $a$ of $\PHI^\dagger$ that consist of two literals from $\cF_\infty\cup\neg\cF_\infty$.
Let $\cC_\infty$ be the set of such clauses $a$.
For each clause $a\in\cC_\infty$ the graph $\cG$ contains two edges joining $a$ and its two constituent literals.
The graph $G(\PHI^*)$ that we are ultimately interested in results from $\cG$ by contracting pairs of inverse literals $l,\neg l\in \cF_\infty\cup\neg\cF_\infty$.

A large excess $\vm^*-\vn^*$ can either caused by the presence of atypically many clauses in $\cG$ or by excess pairs of inverse litetals that get contracted.
We first address the gain in clauses due to inclusion of $\neg\cF_\infty$ and all induced clauses.
The exploration process discovers each literal $\lambda\in \cF_\infty\setminus\{x_1,\neg x_1,\ldots,x_\ell,\neg x_\ell\}$ via a clause $\neg l_t\vee \lambda$, {where $\neg l_t \in \cE_{t-1}$}.
Thus, $|\cC_\infty|\geq|\cF_\infty|-2\ell$.
Hence, the random variable $\vX=|\cC_\infty|-|\cF_\infty|+2\ell$ accounts for the number of excess clauses that are present among the literals $\cF_\infty\cup\neg\cF_\infty$ but that were not probed by the process.
We highlight that $\vX$ also counts clauses that contain two literals from the seed set $\{x_1,\neg x_1,\ldots,x_\ell,\neg x_\ell\}$.
Because clauses appear in $\PHI^\dagger$ independently with probability $p=O(d/n)$, we obtain the bounds
\begin{align}\label{eqCase1}
\pr[\vX\geq1\mid |\cF_\infty|]&\leq O(|\cF_\infty|^2/n),&
\pr[\vX\geq2\mid |\cF_\infty|]&\leq O(|\cF_\infty|^4/n^2).
\end{align}

Secondly, we investigate the loss in nodes due to contraction.
Hence, $\vn^*=|\cF_\infty\cup\neg\cF_\infty|/2$.
By construction, the seeds $x_1,\neg x_1,\ldots,x_\ell,\neg x_\ell$ come in pairs.
Let $\vX'=\frac12|\cF_\infty\cap\neg\cF_\infty|-\ell$ count the number of excess inverse literal pairs that we need to contract.
Since the process is oblivious to the identities of the variables underlying the literals, given its size the set $\cF_\infty\setminus\{x_1,\neg x_1,\ldots,x_\ell,\neg x_\ell\}$ is a uniformly random subset of the set $\{x_i,\neg x_i:\ell <i\leq n\}$ of non-seed literals.
Therefore, a routine balls-into-bins argument shows that
\begin{align}\label{eqCase2}
\pr[\vX'\geq1\mid |\cF_\infty|]&\leq O(|\cF_\infty|^2/n),&
\pr[\vX'\geq2\mid |\cF_\infty|]&\leq O(|\cF_\infty|^4/n^2).
\end{align}

Finally, in order to estimate $\vm^*-\vn^*$ we consider four separate cases.
\begin{description}
\item[Case 1: $\vX = \vX' = 0$]
Since $\vX=0$ the graph $\cG$ is a forest with $2\ell$ components rooted at $x_1,\neg x_1,\ldots,x_\ell,\neg x_\ell$.
Moreover, since $\vX'=0$ we have $\cF_\infty\cap\neg\cF_\infty=\{x_1,\neg x_1,\ldots,x_\ell,\neg x_\ell\}$.
Therefore, $G(\PHI^*)$ is obtained from $\cG$ by identifying the pairs $x_i,\neg x_i$ for $i=1,\ldots,\ell$.
Hence, $G(\PHI^*)$ is a forest with $\ell$  components, and thus
\begin{align}\label{eqCase_1}
\vm^*=\vn^*-\ell.
\end{align}
\item[Case 2: $\vX = 1$, $\vX'=0$]
Obtain $\hat\cG$ from $\cG$ by adding one new vertex $r$ whose neighbours are $x_1,\neg x_1,\ldots,x_\ell,\neg x_\ell$.
Then $\hat\cG$ is unicyclic because $\vX=1$.
Let $\tilde\cG$ be the graph obtained from $\hat\cG$ by deleting the vertex $r$ along with one (arbitrary) clause $a$ from the cycle of $\hat\cG$.
Then $\tilde\cG$ is a forest with $2\ell$ components.
Therefore, by the same token as in Case~1, $G(\PHI^*-a)$ is a forest with $\ell$ components. 
Hence, $G(\PHI^*)$, obtained by inserting clause $a$ into  $G(\PHI^*-a)$, either contains a single cycle or consists of exactly $\ell-1$ components.
Thus, by \eqref{eqCase1}
\begin{align}\label{eqCase_2}
\vm^*&\leq\vn^*-\ell+1,&\pr\brk{\vX=1,\,\vX'=0\mid|\cF_\infty|}&=O(|\cF_\infty|^2/n).
\end{align}
\item[Case 3: $\vX = 0$, $\vX'=1$]
The graph $\hat\cG$, defined as in Case~2, is a tree because $\vX=0$.
Suppose $(\cF_\infty\cap\neg\cF_\infty)\setminus\{x_1,\neg x_1,\ldots,x_\ell,\neg x_\ell\}=\{y,\neg y\}$.
Let $a$ be a clause on the unique path from $y$ to $\neg y$ in $\hat\cG$.
Then the same argument as in Case~1 shows that  $G(\PHI^*-a)$ is a forest with $\ell$ components.
Therefore, $G(\PHI^*)$ either contains a unique cycle or has precisely $\ell-1$ components.
Consequently, \eqref{eqCase2} yields
\begin{align}\label{eqCase_3}
\vm^*&\leq\vn^*-\ell+1,&\pr\brk{\vX=0,\,\vX'=1\mid|\cF_\infty|}&=O(|\cF_\infty|^2/n).
\end{align}
\item[Case 4: $\vX  + \vX' \geq 2$]
In this case we do not have a bound on $\vm^*-\vn^*$, but we claim that
\begin{align}\label{eqCase_4}
\pr\brk{\vX+\vX'\geq2\mid|\cF_\infty|}&=O(|\cF_\infty|^4/n^2).
\end{align}
Indeed, \eqref{eqCase1} and \eqref{eqCase2} readily imply that  $\pr\brk{\vX\vee\vX'\geq2\mid|\cF_\infty|}=O(|\cF_\infty|^4/n^2)$.
Further, since $\vX$ is independent of $\vX'$ given $\cF_\infty$, \eqref{eqCase1} and \eqref{eqCase2} also yield the bound 
$\pr\brk{\vX=\vX'=1\mid|\cF_\infty|}=O(|\cF_\infty|^4/n^2)$.
\end{description}
The assertion follows by combining \eqref{eqCase_1}--\eqref{eqCase_4} with Claim~\ref{Lemma_size_of_implication_graphs}. 
\end{proof}

\begin{claim}\label{Claim_A}
For all $\chi\in\{\pm1\}^{\{x_1,\ldots,x_\ell\}}$ we have
$\vA_\chi^\dagger \leq|\cF_\infty|\vecone\cbc{\vm^*\leq\vn^*-\ell+1}+n\vecone\cbc{\vm^*>\vn^*-\ell+1}$.
\end{claim}
\begin{proof}
The graph $G(\PHI^*)$ consists of at most $\ell$ components (one for each of the initial variables $x_1,\ldots,x_\ell$).
Hence, $\vm^*\geq\vn^*-\ell$ and $G(\PHI^*)$ is acyclic if $\vm^*=\vn^*-\ell$.
Moreover, if $G(\PHI^*)$ is acyclic then $\vA_\chi^\dagger\leq|\cF_\infty|$ by construction.

Thus, we are left to consider the case $\vm^*=\vn^*-\ell+1$.
Then $\PHI^*$ contains a clause $a$ such that $G(\PHI^*-a)$ is a forest with $\ell$ components rooted at $x_1,\ldots,x_\ell$.
Assume without loss that $a=x_{n-1}\vee x_n$.
Then by construction we have $\{x_{n-1},\neg x_{n-1}\}\cap\cF_\infty\neq\emptyset$ and $\{x_n,\neg x_n\}\cap\cF_\infty\neq\emptyset$.
Further, unless $\neg x_{n-1},\neg x_n\in\cF_\infty$ we have $\vA_\chi\leq\vI_\chi\leq|\cF_\infty|$ as in the first case.
Hence, assume that $\neg x_{n-1},\neg x_n\in\cF_\infty$.
Let $i\in[\ell]$ be such that $x_n$ belongs to the connected component of $x_i$ in $G(\PHI^*-a)$ and obtain $\chi'$ from $\chi$ by flipping the value assigned to $x_i$.
Because $G(\PHI^*-a)$ is a forest, we conclude that $\vA_\chi^\dagger\leq\vI_\chi\wedge\vI_{\chi'}\leq|\cF_\infty|$.
\end{proof}

\begin{proof}[Proof of \Lem~\ref{Lemma_tame}]
The choice of the clause probability $p$ ensures that $\vA_\chi^\dagger$ stochastically dominates $\vA_\chi$.
Therefore, the assertion follows from Claims~\ref{Lemma_size_of_implication_graphs}--\ref{Claim_A}.
\end{proof}

\begin{proof}[Proof of \Lem~\ref{Lemma_wild}]
The choice of the clause probability $p$ and the construction of the set $\cF_\infty$ guarantee that $\vI_\chi$ is stochastically dominated by the random variable $|\cF_\infty|\vecone\cbc{\vm^*\leq\vn^*-\ell}+n\vecone\cbc{\vm^*>\vn^*-\ell}$.
Hence, Claims~\ref{Lemma_size_of_implication_graphs}--\ref{Claim_A} imply the desired bound.
\end{proof}

\begin{proof}[{Proof of \Cor~\ref{Lemma_emp}}]
Let $Y=\{x_1\}$ and $\chi^+_{x_1}=1$, $\chi^-_{x_1}=-1$.
Assume that $\PHI'$ is satisfiable.
Then Fact~\ref{Fact_UC} implies that 
\begin{align*}
n\wedge\abs{\log\frac{\mu_{\PHI'}(\SIGMA_{x_1}=1)}{\mu_{\PHI'}(\SIGMA_{x_1}=-1)}}\leq \vI_{\chi^-}+\vI_{\chi^+}.
\end{align*}
Therefore, the assertion follows from \Lem~\ref{Lemma_wild}.
\end{proof}

\subsection{Proof of \Prop~\ref{Lemma_PHI''}}
Let $c_1,\ldots,c_{\vec\Delta''}$ be the new clauses added to $\PHI''$ and let $\vec Y=\{\vy_1,\vz_1,\ldots,\vy_{\DELTA''},\vz_{\DELTA''}\}$ be the set of variables that occur in these clauses.
We begin by deriving the following rough bound.

\begin{lemma}\label{Lemma_Difference_Z}
We have
$\displaystyle\Erw \brk{\log^{2} \frac{Z(\PHI'')\vee1}{Z(\PHI')\vee1}} = O(1).$
\end{lemma}
\begin{proof}
If $\PHI'$ is unsatisfiable then so is $\PHI''$ and thus $(Z(\PHI'')\vee1)/(Z(\PHI')\vee1)=1$.
Hence, we may assume that $Z(\PHI')\geq1$.
If $|\vY|=2\vec\Delta''$, the new clauses attach to disjoint sets of variables.
Consider the truth value assignment $\CHI\in\{\pm1\}^{\vY}$ that satisfies both literals in each of the clauses $c_1,\ldots,c_{\vec\Delta''}$.
Fact~\ref{Fact_UC} shows that
\begin{align}\label{eqLemma_Difference_Z_1}
Z(\PHI'')\vee1&\geq Z(\PHI',\CHI)\vee1\geq 2^{-\vA_{\CHI}}Z(\PHI').
\end{align}
Combining \eqref{eqLemma_Difference_Z_1} with \Lem~\ref{Lemma_tame} and recalling that $\DELTA''\disteq\Po(d/2)$, we obtain
\begin{align}\label{eqLemma_Difference_Z_2}
\Erw\brk{\vecone\cbc{Z(\PHI')\geq1,\ |\vY|=2\DELTA''}\log^2\frac{Z(\PHI'')\vee1}{Z(\PHI')\vee
1}}&\leq \Erw\brk{\vecone\cbc{Z(\PHI')\geq1,\ |\vY|=2\DELTA''}\vA_{\CHI}^2}=O(1).
\end{align}

Next, consider the event $|\vY|=2\vec\Delta''-1$.
Because $c_1,\ldots,c_{\vec\Delta''}$ are drawn independently, we have
\begin{align}	\label{eqLemma_Difference_Z_3a}
\pr\brk{|\vY|=2\vec\Delta''-1\mid\DELTA''}&\leq O((\vec\Delta'')^2/n).
\end{align}
Moreover, 
because the signs of the clauses $c_1,\ldots,c_{\DELTA''}$ are independent of $\PHI'$,
given $|\vY|=2\vec\Delta-1$ there exists an assignment $\CHI\in\{\pm1\}^{\vY}$, stochastically independent of $\PHI'$, that satisfies $c_1,\ldots,c_{\vec\Delta''}$.
Fact~\ref{Fact_UC} yields $Z(\PHI'')\vee1\geq Z(\PHI',\CHI)\geq2^{-\vI_{\CHI}}Z(\PHI')$.
Therefore, since $\log((Z(\PHI'')\vee1)/(Z(\PHI')\vee1))\leq n$, \Lem~\ref{Lemma_wild} and \eqref{eqLemma_Difference_Z_3a} imply
\begin{align}
\Erw\brk{\vecone\cbc{Z(\PHI')\geq1,\ |\vY|=2\vec\Delta''-1}\log^2\frac{Z(\PHI'')\vee1}{Z(\PHI')\vee1}}
&\leq n\Erw\brk{\vecone\cbc{|\vY|=2\vec\Delta''-1}\vI_{\CHI}}
	= O(1).
	\label{eqLemma_Difference_Z_3}
\end{align}

Finally, consider the event $|\vY|<2\vec\Delta''-1$.
Due to the independence of $c_1,\ldots,c_{\vec\Delta''}$, this event occurs with probability $O(n^{-2})$.
Hence, the deterministic bound $(Z(\PHI'')\vee1)/(Z(\PHI')\vee1)\geq2^{-n}$ implies
\begin{align}\label{eqLemma_Difference_Z_4}
\Erw\brk{\vecone\cbc{Z(\PHI')\geq1,\ |\vY|<2\vec\Delta''-1}\log^2\frac{Z(\PHI'')\vee1}{Z(\PHI')\vee1}}=O(1).
\end{align}
The assertion follows from \eqref{eqLemma_Difference_Z_2}, \eqref{eqLemma_Difference_Z_3} and \eqref{eqLemma_Difference_Z_4}.
\end{proof}

\begin{lemma}\label{Lemma_Lambda_bound}
There exists a number $K>0$ such that for every $\eps>0$ we have
$$\limsup_{n\to\infty}\Erw\brk{\bc{\sum_{i=1}^{\vec\Delta''}\Lambda_\eps(1-\mu_{\PHI'}(\SIGMA_{\vec y_i}=-\sign(c_i,\vec y_i))\mu_{\PHI'}(\SIGMA_{\vz_i}=-\sign(c_i,\vec z_i))}^2\mid Z(\PHI')>0}\leq K.$$
\end{lemma}
\begin{proof}
Since $\DELTA"\disteq\Po(d/2)$ and the pair $(\vy_1,\vz_1)$ is uniformly random, due to Cauchy-Schwarz it suffices to prove 
$\limsup_{n\to\infty}\Erw\brk{\Lambda_\eps(1-\mu_{\PHI'}(\SIGMA_{x_1}=1)\mu_{\PHI'}(\SIGMA_{x_2}=1))^2\mid Z(\PHI')>0}\leq K$
for every $\eps>0$.
We observe that
\begin{align}\nonumber
\limsup_{n\to\infty}\ &\Erw\brk{\Lambda_\eps(1-\mu_{\PHI'}(\SIGMA_{x_1}=1)\mu_{\PHI'}(\SIGMA_{x_2}=1))^2\mid Z(\PHI')>0}
\leq \limsup_{n\to\infty}\Erw\brk{\Lambda_\eps(1-\mu_{\PHI'}(\SIGMA_{x_1}=1))^2\mid Z(\PHI')>0}\\
&=\limsup_{n\to\infty}\Erw\brk{\frac1n\sum_{i=1}^n\Lambda_\eps(1-\mu_{\PHI'}(\SIGMA_{x_i}=1))^2\mid Z(\PHI')>0}.
\label{eqLemma_Lambda_bound_1}
\end{align}
Moreover, $\PHI'$ has $\vm'\disteq\Po(dn/2-d/2)$ clauses, while $\PHI=\PHI_n$ has $\vm\disteq\Po(dn/2)$ clauses.
Since $\dTV(\vm',\vm)=o(1)$, the formulas $\PHI'$, $\PHI$ can be coupled such that both coincide \whp{}
Hence, for any fixed $\eps>0$ we have
\begin{align}\label{eqLemma_Lambda_bound_2}
\Erw\brk{\frac1n\sum_{i=1}^n\Lambda_\eps(1-\mu_{\PHI'}(\SIGMA_{x_i}=1))^2\mid Z(\PHI')}=
\Erw\brk{\frac1n\sum_{i=1}^n\Lambda_\eps(1-\mu_{\PHI}(\SIGMA_{x_i}=1))^2\mid Z(\PHI')}+o(1).
\end{align}
Further,  since for every $\eps>0$ the function $u\in[0,1]\mapsto \Lambda_\eps(1-u)^2$ is  continuous, \Cor~\ref{Cor_BP} implies that
\begin{align}\label{eqLemma_Lambda_bound_3}
\frac1n\sum_{i=1}^n\Lambda_\eps(1-\mu_{\PHI}(\SIGMA_{x_i}=1))^2&\quad\stacksign{$n\to\infty$}{\longrightarrow}\quad
\Erw\brk{\Lambda_\eps(1-\MU_{\pi_d})^2}&&\mbox{in probability}.
\end{align}
Since \eqref{eqProp_uniqueness2} shows that
 $\Erw\brk{\Lambda_\eps(1-\MU_{\pi_d})^2}\leq\Erw\brk{\log^2(1-\MU_{\pi_d})}<\infty$, 
the assertion follows from \eqref{eqLemma_Lambda_bound_1}--\eqref{eqLemma_Lambda_bound_3}.
\end{proof}

\begin{lemma}\label{Lemma_PHI''a}
For any $\delta>0$ there exists $\eps>0$ such that
\begin{align*}
\limsup_{n\to\infty}\abs{\Erw\brk{\log\frac{Z(\PHI'')\vee1}{Z(\PHI')\vee1}}-\frac{d}{2} \Erw\brk{\Lambda_\eps \bc{ 1-\mu_{\PHI',x_{1}}(\vs_1)\mu_{\PHI',x_{2}}(\vs_2) }\mid Z(\PHI')}}&<\delta.
\end{align*}
\end{lemma}
\begin{proof}
Choose small enough $\xi=\xi(\delta)>\eta=\eta(\xi)>\eps=\eps(\eta)>0$, assume that $n>n_0(\eps)$ is sufficiently large and let $(\gamma_n)_n$ be a sequence of positive reals, depending on $\xi$ and $\eta$, that tends to zero sufficiently slowly.
Let $\cE=\cE_n$ be the event that the following five statements hold.
\begin{description}
\item[E1] $Z(\PHI')>0$.
\item[E2] $|\vY|=2\vec\Delta''$.
\item[E3] $\vec\Delta''<\xi^{-1/4}$.
\item[E4] for all $y\in\vY$ and all $s\in\cbc{\pm1}$ we have $\mu_{\PHI'}(\SIGMA_y=s)<1-2\eta.$
\item[E5] $\sum_{\sigma\in\{\pm1\}^{\vY}}\abs{\mu_{\PHI}(\forall y\in\vY:\SIGMA_{y}=\sigma_y)-\prod_{y\in\vY}\mu_{\PHI}(\SIGMA_{y}=\sigma_y)}<\gamma_n$.
\end{description}
The first two events {\bf E1}, {\bf E2} occur with probability $1-o(1)$ as $n\to\infty$.
Moreover, $\pr[\mbox{\bf E3}]>1-\xi$ if $\xi$ is small enough.
Further, since \Cor~\ref{Cor_BP} shows that $\pi_{\PHI}$ converges to $\pi_d$ weakly in probability, the tail bound \eqref{eqProp_uniqueness2} implies that
$\pr\brk{\mbox{\bf E4}\mid \vec\Delta''<\xi^{-1/4}}>1-\xi,$
provided that $\eta$ is small enough.
Additionally, \Cor~\ref{Cor_uniqueness} implies $\pr\brk{\mbox{\bf E5}\mid\mbox{\bf E1--E4}}=1-o(1)$ if $\gamma_n\to0$ slowly enough.
Consequently, 
\begin{align}\label{eqcE}
\pr\brk{\cE}&>1-4\xi.
\end{align}
Combining \Lem~\ref{Lemma_Difference_Z}, \eqref{eqcE} and the Cauchy-Schwarz inequality, we obtain
\begin{align}\label{eqcE1}
\abs{\Erw\brk{(1-\vecone\cE)\log\frac{Z(\PHI'')}{Z(\PHI')}}}\leq\delta/3+o(1).
\end{align}
Similarly, by \Lem~\ref{Lemma_Lambda_bound}, \eqref{eqcE} and Cauchy-Schwarz,
\begin{align}\label{eqcE2}
\abs{
\Erw\brk{(1-\vecone\cE)\sum_{i=1}^{\vec\Delta''}
		\Lambda_\eps\bc{1-\mu_{\PHI'}(\SIGMA_{\vy_i}=-\sign(\vy_i,c_i))
	\mu_{\PHI'}(\SIGMA_{\vz_i}=-\sign(\vz_i,c_i))}}}
\leq\delta/3+o(1).
\end{align}
Further, because the distribution of $\PHI'$ is invariant under permutations of the variables $x_1,\ldots,x_n$ and $\Erw[\DELTA'']=d/2$,
\begin{align}\nonumber
\Erw&\brk{\sum_{i=1}^{\vec\Delta''}\Lambda_\eps\bc{1-\mu_{\PHI'}(\SIGMA_{\vy_i}=-\sign(\vy_i,c_i))
	\mu_{\PHI'}(\SIGMA_{\vz_i}=-\sign(\vz_i,c_i))}\mid Z(\PHI')>0}\\&=
	\frac{d}{2} \Erw\brk{\Lambda_\eps \bc{ 1-\mu_{\PHI'}(\SIGMA_{x_1}=\vs_1)\mu_{\PHI'}(\SIGMA_{x_2}=\vs_2) }\mid Z(\PHI')>0}.\label{eqcE3}
\end{align}
Moreover, on the event $\cE$ we have
\begin{align*}
\frac{Z(\PHI'')}{Z(\PHI')}&
	=\sum_{\sigma\in\cbc{\pm1}^{\vY}}\vecone\cbc{\sigma\mbox{ satisfies }c_1,\ldots,c_{\vec\DELTA''}}
			\mu_{\PHI'}\bc{\forall y\in\vY:\SIGMA_y=\sigma_y}\\
&=\sum_{\sigma\in\cbc{\pm1}^{\vY}}\vecone\cbc{\sigma\mbox{ satisfies }c_1,\ldots,c_{\vec\DELTA''}}
	\prod_{y\in\vY}\mu_{\PHI'}\bc{\SIGMA_y=\sigma_y}+o(1)&&\mbox{[due to {\bf E3,E5}]}\\
&=\prod_{i=1}^{\DELTA''}\bc{1-\mu_{\PHI'}\bc{\SIGMA_{\vy_i}=-\sign(\vy_i,c_i)}\mu_{\PHI'}\bc{\SIGMA_{\vz_i}=-\sign(\vz_i,c_i)}}+o(1).
\end{align*}
Therefore, by {\bf E4}
\begin{align}\nonumber
\Erw\brk{\vecone\cE\log\frac{Z(\PHI'')}{Z(\PHI')}}
&=\Erw\brk{\vecone\cE\sum_{i=1}^{\vec\Delta''}
		\log\bc{1-\mu_{\PHI'}\bc{\SIGMA_{\vy_i}=-\sign(\vy_i,c_i)}\mu_{\PHI'}\bc{\SIGMA_{\vz_i}=-\sign(\vz_i,c_i)}}}+o(1)\\
&=\Erw\brk{\vecone\cE\sum_{i=1}^{\vec\Delta''}
		\Lambda_\eps\bc{1-\mu_{\PHI'}\bc{\SIGMA_{\vy_i}=-\sign(\vy_i,c_i)}\mu_{\PHI'}\bc{\SIGMA_{\vz_i}=-\sign(\vz_i,c_i)}}}+o(1).
		\label{eqLemma_PHI''a_1}
\end{align}
Finally, the assertion follows from \eqref{eqcE1}--\eqref{eqLemma_PHI''a_1}.
\end{proof}

\begin{proof}[Proof of \Prop~\ref{Lemma_PHI''}]
\Prop~\ref{Prop_BP} shows that $\MU_{\pi_d,1}$ and $1-\MU_{\pi_d,1}$ are identically distributed.
Since $\Lambda_\eps$ is continuous and bounded, \Cor~\ref{Cor_BP} therefore implies that
\begin{align}\nonumber
\lim_{n\to\infty}\Erw\brk{\Lambda_\eps \bc{ 1-\mu_{\PHI',x_{1}}(\vs_1)\mu_{\PHI',x_{2}}(\vs_2) }}&=
\Erw\brk{\Lambda_\eps \bc{ 1-\bc{\frac{1-\vs_1}{2}+\vs_1\MU_{\pi_d,1}}\bc{\frac{1-\vs_2}{2}+\vs_2\MU_{\pi_d,2}}}}\\
&=\Erw\brk{\Lambda_\eps \bc{ 1-\MU_{\pi_d,1}\MU_{\pi_d,2}}}. \label{eqLemma_PHI''_1}
\end{align}
for every $\eps>0$.
Further, since $\Lambda_\eps(1-\MU_{\pi_d,1}\MU_{\pi_d,2})$ decreases monotonically to  $\log(1-\MU_{\pi_d,1}\MU_{\pi_d,2})$ as $\eps\to0$, the monotone convergence theorem and \eqref{eqBFE} yield
\begin{align}\label{eqLemma_PHI''_2}
\lim_{\eps\to0}\Erw\brk{\Lambda_\eps \bc{ 1-\MU_{\pi_d,1}\MU_{\pi_d,2}}}=\Erw\log\bc{ 1-\MU_{\pi_d,1}\MU_{\pi_d,2}}.
\end{align}
Combining \eqref{eqLemma_PHI''_1} and~\eqref{eqLemma_PHI''_2} and \Lem~\ref{Lemma_PHI''a} completes the proof.
\end{proof}

\subsection{Proof of \Prop~\ref{Lemma_PHI'''}}
The steps that we follow are analogous to the ones from the proof of \Prop~\ref{Lemma_PHI''}.
Recall that $\PHI'''$ is obtained from $\PHI'$ by adding one variable $x_{n+1}$ along with random adjacent clauses $b_1,\ldots,b_{\DELTA'''}$, where $\DELTA'''$ is a Poisson variable with mean $d$.
Let $\vy_1,\ldots,\vy_{\DELTA'''}\in\{x_1,\ldots,x_n\}$ be the variables of $\PHI'$ where the new clauses attach and let $\vY=\{\vy_1,\ldots,\vy_{\DELTA'''}\}$.
We begin with the following $L_2$-bound.

\begin{lemma}\label{Lemma_Difference_Z'''}
We have $\displaystyle\limsup_{n\to\infty}\Erw \brk{\log^{2} \frac{Z(\PHI''')\vee1}{Z(\PHI')\vee1}}<\infty.$
\end{lemma}
\begin{proof}
If $\PHI'$ is unsatisfiable, then so is $\PHI'''$ and thus $(Z(\PHI''')\vee1)/(Z(\PHI')\vee1)=1$.
Hence, we may assume that $Z(\PHI')\geq1$.
We now consider three scenarios.
First, suppose that $|\vY|=\vec\Delta'''$, i.e., the new clauses attach to distinct variables of $\PHI'$.
Then define an assignment $\CHI\in\{\pm1\}^{\vY}$ by setting each $y\in\vY$ to the value that satisfies the unique clause among $b_1,\ldots,b_{\DELTA'''}$ in which $y$ occurs.
We claim that
\begin{align}\label{eqLemma_Difference_Z'''_1}
Z(\PHI''')\vee1&\geq2^{-\vA_{\CHI}}Z(\PHI').
\end{align}
Indeed, if $\CHI'\in\{\pm1\}^{\vY}$ differs from $\CHI$ on only one variable, then we can always satisfy all clauses $b_1,\ldots,b_{\DELTA'''}$ by setting $x_{n+1}$ appropriately.
Therefore, \eqref{eqLemma_Difference_Z'''_1} follows from Fact~\ref{Fact_UC} and the definition \eqref{eqAchi} of $\vA_{\CHI}$.
Combining \eqref{eqLemma_Difference_Z'''_1} with \Lem~\ref{Lemma_tame}, we obtain
\begin{align}\label{eqLemma_Difference_Z'''_2}
\Erw\brk{\vecone\cbc{Z(\PHI')\geq1,\ |\vY|=\vec\Delta'''}\log^2\frac{Z(\PHI''')\vee1}{Z(\PHI')\vee
1}}&\leq \Erw\brk{\vecone\cbc{|\vY|=\vec\Delta'''}\vA_{\CHI}^2}=O(1).
\end{align}

Second, consider the case $|\vY|=\vec\Delta'''-1$.
Because $b_1,\ldots,b_{\vec\Delta'''}$ are drawn independently, we have
\begin{align}\label{eqLemma_Difference_Z'''_3a}
\pr\brk{|\vY|=\vec\Delta'''-1\mid\DELTA'''}&=O((\vec\Delta''')^2/n).
\end{align}
Further, there exists an assignment $\CHI\in\{\pm1\}^{\vY}$ under which all but one of the clauses $b_1,\ldots,b_{\vec\Delta'''}$ are satisfied.
This assignment is independent of $\PHI'$ because the signs of $b_1,\ldots,b_{\DELTA'''}$ are.
Since we can use the new variable $x_{n+1}$ to satisfy the last clause as well, Fact~\ref{Fact_UC} implies the bound $(Z(\PHI''')\vee1)/Z(\PHI')\geq2^{-\vI_{\CHI}}$.
Therefore, \Lem~\ref{Lemma_wild} and \eqref{eqLemma_Difference_Z'''_3a} yield
\begin{align}
\Erw\brk{\vecone\cbc{Z(\PHI')\geq1,\ |\vY|=\vec\Delta'''-1}
	\log^2\frac{Z(\PHI'')\vee1}{Z(\PHI')\vee1}}
	&\leq \Erw\brk{\vecone\cbc{Z(\PHI')\geq1,\ |\vY|=\vec\Delta'''-1}\vI_{\CHI}^2}\nonumber\\
&\leq n\Erw\brk{\vecone\cbc{|\vY|=\vec\Delta'''-1}\vI_{\CHI}}=O(1).
	\label{eqLemma_Difference_Z'''_3}
\end{align}

Finally, because $b_1,\ldots,b_{\vec\Delta'''}$ are drawn independently, the event $\{|\vY|<\DELTA'''-1\}$ has probability $O(n^{-2})$.
Therefore, the deterministic bound $(Z(\PHI''')\vee1)/(Z(\PHI')\vee1)\geq2^{-n}$ ensures that
\begin{align}\label{eqLemma_Difference_Z'''_4}
\Erw\brk{\vecone\cbc{Z(\PHI')\geq1,\ |\vY|<\vec\Delta'''-1}
\log^2\frac{Z(\PHI''')\vee1}{Z(\PHI')\vee1}}=O(1).
\end{align}
The assertion follows from \eqref{eqLemma_Difference_Z'''_2}, \eqref{eqLemma_Difference_Z'''_3} and \eqref{eqLemma_Difference_Z'''_4}.
\end{proof}

\begin{lemma}\label{Lemma_Lambda_bound'''}
There exists $K>0$ such that for every $\eps>0$ we have
\begin{align*}
\limsup_{n\to\infty}
\Erw\brk{\Lambda_\eps
	\bc{\sum_{s\in\{\pm1\}}\prod_{i=1}^{\vec\Delta'''}\bc{1-\vecone\cbc{s\neq\sign(x_{n+1},b_i)}
			\mu_{\PHI'}(\SIGMA_{\vy_i}=-\sign(\vy_i,b_i))}}^2\mid Z(\PHI')>0}\leq K.
\end{align*}
\end{lemma}
\begin{proof}
Since $\DELTA'''\disteq\Po(d/2)$, $\vy_1,\ldots,\vy_{\DELTA'''}$ and the signs $\sign(b_i,\vy_i)$ are uniformly random, we obtain
\begin{align}\nonumber
\Erw&\brk{\Lambda_\eps
	\bc{\sum_{s\in\{\pm1\}}\prod_{i=1}^{\vec\Delta'''}\bc{1-\vecone\cbc{s\neq\sign(x_{n+1},b_i)}
			\mu_{\PHI'}(\SIGMA_{\vy_i}=-\sign(\vy_i,b_i))}}^2\mid Z(\PHI')>0}\\
&\leq 1+\Erw\brk{\Lambda_\eps
	\bc{\prod_{i=1}^{\vec\Delta'''}\mu_{\PHI'}(\SIGMA_{\vy_{i}}=1)}^2\mid Z(\PHI')>0}\leq 1+d\Erw\brk{\Lambda_\eps\bc{\mu_{\PHI'}(\SIGMA_{\vy_1}=1)}^2\mid Z(\PHI')>0}.\label{eqLemma_Lambda_bound_1'''}
\end{align}
Further,  the formulas $\PHI'$, $\PHI$ can be coupled such that both coincide \whp\ (cf.\ the proof of \Lem~\ref{Lemma_Lambda_bound}).
Therefore, \Cor~\ref{Cor_BP} implies that for every $\eps>0$,
\begin{align}\label{eqLemma_Lambda_bound_2'''}
\Erw\brk{\Lambda_\eps\bc{\mu_{\PHI'}(\SIGMA_{\vy_1}=1)}^2\mid Z(\PHI')>0}=\Erw\brk{\Lambda_\eps\bc{\MU_{\pi_{\PHI}}}^2\mid Z(\PHI)>0}+o(1)
	=\Erw\brk{\Lambda_\eps\bc{\MU_{\pi_{d}}^2}}+o(1)\leq\Erw\brk{\log^2\MU_{\pi_{d}}}+o(1).
\end{align}
Since \eqref{eqProp_uniqueness2} implies that $\Erw\brk{\log^2\MU_{\pi_{d}}}<\infty$, the assertion follows from \eqref{eqLemma_Lambda_bound_1'''}--\eqref{eqLemma_Lambda_bound_2'''}.
\end{proof}

\begin{lemma}\label{Lemma_PHI'''a}
For any $\delta>0$ there exists $\eps_0>0$ such that for every $0<\eps<\eps_0$,
\begin{align*}
\abs{\Erw\brk{\log\frac{Z(\PHI''')\vee1}{Z(\PHI')\vee1}}-
	\Erw\brk{\Lambda_\eps
	\bc{\sum_{s\in\{\pm1\}}\prod_{i=1}^{\vd}\bc{1-\vecone\cbc{s\neq\vs_i}
			\mu_{\PHI'}(\SIGMA_{x_i}=\vs_{i}')}}\mid Z(\PHI')>0}}&<\delta + o(1).
\end{align*}
\end{lemma}
\begin{proof}
Choose small enough $\xi=\xi(\delta)>\eta=\eta(\xi)>\eps=\eps(\eta)>0$, assume that $n>n_0(\eps)$ is sufficiently large and let $(\gamma_n)_n$ be a sequence of numbers $\gamma_n>0$ that tends to zero slowly.
Let $\cE=\cE_n$ be the event that the following five statements are satisfied.
\begin{description}
\item[E1] $Z(\PHI')>0$.
\item[E2] $|\vY|=\vec\Delta'''$.
\item[E3] $\vec\Delta'''<\xi^{-1/4}$.
\item[E4] for all $y\in\vY$ we have $\mu_{\PHI'}(\SIGMA_y=1)\vee\mu_{\PHI'}(\SIGMA_y=-1)<1-2\eta.$
\item[E5] $\sum_{\sigma\in\{\pm1\}^{\vY}}\abs{\mu_{\PHI}(\forall y\in\vY:\SIGMA_{y}=\sigma_y)-\prod_{y\in\vY}\mu_{\PHI}(\SIGMA_{y}=\sigma_y)}<\gamma_n$.
\end{description}
As in the proof of \Lem~\ref{Lemma_PHI''a} we obtain $\pr\brk{\cE}>1-4\xi.$
Hence, \Lem s~\ref{Lemma_Difference_Z'''} and~\ref{Lemma_Lambda_bound'''}  and the Cauchy-Schwarz inequality yield
\begin{align}\label{eqcE1'''}
\abs{\Erw\brk{(1-\vecone\cE)\log\frac{Z(\PHI''')}{Z(\PHI')}}}&\leq\delta/3+o(1),\\
\abs{
\Erw\brk{(1-\vecone\cE)\Lambda_\eps
	\bc{\sum_{s\in\{\pm1\}}\prod_{i=1}^{\vec\Delta'''} \bc{1-\vecone\cbc{s\neq\sign(x_{n+1},b_i)}
			\mu_{\PHI'}(\SIGMA_{\vy_i}=-\sign(\vy_i,b_i))}}\mid Z(\PHI')>0}}
&\leq\delta/3+o(1).\label{eqcE2'''}
\end{align}
Moreover, because the distribution of $\PHI'$ is invariant under variable permutations,
\begin{align}\nonumber
\Erw&\brk{
\Lambda_\eps
	\bc{\sum_{s\in\{\pm1\}}\prod_{i=1}^{\vec\Delta'''}\bc{1-\vecone\cbc{s\neq\sign(x_{n+1},b_i)}
			\mu_{\PHI'}(\SIGMA_{\vy_i}=-\sign(\vy_i,b_i))}}\mid Z(\PHI')>0}\\
			&=\Erw\brk{\Lambda_\eps
	\bc{\sum_{s\in\{\pm1\}}\prod_{i=1}^{\vd}\bc{1-\vecone\cbc{s\neq\vs_i}
			\mu_{\PHI'}(\SIGMA_{x_i}=\vs_{i}')}}\mid Z(\PHI')>0}
	+o(1).\label{eqcE3'''}  
\end{align}
Further, on $\cE$ we obtain
\begin{align}\nonumber
\frac{Z(\PHI''')}{Z(\PHI')}&
	=\sum_{\sigma\in\cbc{\pm1}^{\vY\cup\{x_{n+1}\}}}\vecone\cbc{\sigma\mbox{ satisfies }b_1,\ldots,b_{\vec\DELTA'''}}
			\mu_{\PHI'}\bc{\forall y\in\vY:\SIGMA_y=\sigma_y}\\
&=\sum_{\sigma\in\cbc{\pm1}^{\vY\cup\{x_{n+1\}}}}\vecone\cbc{\sigma\mbox{ satisfies }b_1,\ldots,b_{\vec\DELTA'''}}
	\prod_{y\in\vY}\mu_{\PHI'}\bc{\SIGMA_y=\sigma_y}+o(1)&&\mbox{[due to {\bf E3,E5}]}\nonumber\\
&=
\sum_{s\in\{\pm1\}}\prod_{\substack{i\in[\vec\Delta''']\\\sign(x_{n+1},b_i)=-s}}\mu_{\PHI'}(\SIGMA_{\vy_i}=\sign(\vy_i,b_i))\enspace;
\label{eqLemma_PHI'''a_1}
\end{align}
to elaborate, in the last step $s$ represents the value assigned to $x_{n+1}$ and the product ensures that the clauses $b_i$ in which $x_{n+1}$ occurs with sign $-s$ are satisfied by assigning their second variable $\vy_i$ the value $\sign(\vy_i,b_i)$.
Further, \eqref{eqLemma_PHI'''a_1}, {\bf E3} and {\bf E4} yield
\begin{align}\nonumber
\Erw\brk{\vecone\cE\log\frac{Z(\PHI''')}{Z(\PHI')}}
&=\Erw\brk{\vecone\cE
		\log\bc{\sum_{s\in\{\pm1\}}\prod_{i=1}^{\DELTA'''}\bc{1-\vecone\cbc{\sign(x_{n+1},b_i)=-s}\mu_{\PHI'}(\SIGMA_{\vy_i}=-\sign(\vy_i,b_i))}}}+o(1)\\
&=\Erw\brk{\vecone\cE
		\Lambda_\eps\bc{\sum_{s\in\{\pm1\}}\prod_{i=1}^{\DELTA'''}\bc{1-\vecone\cbc{\sign(x_{n+1},b_i)=-s}\mu_{\PHI'}(\SIGMA_{\vy_i}=-\sign(\vy_i,b_i))}}}+o(1)
		\label{eqLemma_PHI'''a_1a}
\end{align}
Finally, the assertion follows from  \eqref{eqcE1'''}, \eqref{eqcE2'''}, \eqref{eqcE3'''} and~\eqref{eqLemma_PHI'''a_1a}.
\end{proof}

\begin{proof}[Proof of \Prop~\ref{Lemma_PHI''}]
Because $\MU_{\pi_d,1}\disteq 1-\MU_{\pi_d,1}$ by \Prop~\ref{Prop_BP}, \Cor~\ref{Cor_BP} shows that for every $\eps>0$,
\begin{align}\label{eqLemma_PHI'''_1}
	\lim_{n\to\infty}\Erw\brk{\Lambda_\eps
	\bc{\sum_{s\in\{\pm1\}}\prod_{i=1}^{\vd}\bc{1-\vecone\cbc{s\neq\vs_i}\mu_{\PHI'}(\SIGMA_{x_i}=\vs_{i}')}\mid Z(\PHI')>0}
	}&=
	\Erw\brk{\Lambda_\eps\bc{\sum_{s\in\{\pm1\}}\prod_{i=1}^{\vd}\bc{1-\vecone\cbc{s\neq\vs_i}\MU_{\pi_d,i}}}}.
\end{align}
Further, the dominated convergence theorem and \eqref{eqBFE} yield
\begin{align}\label{eqLemma_PHI'''_2}
\lim_{\eps\to0}\Erw\brk{\Lambda_\eps\bc{\sum_{s\in\{\pm1\}}\prod_{i=1}^{\vd}\bc{1-\vecone\cbc{s\neq\vs_i}\MU_{\pi_d,i}}}}
=\Erw\log\bc{\sum_{s\in\{\pm1\}}\prod_{i=1}^{\vd}\bc{1-\vecone\cbc{s\neq\vs_i}\MU_{\pi_d,i}}}.
\end{align}
To complete the proof we combine \eqref{eqLemma_PHI'''_1}, \eqref{eqLemma_PHI'''_2} and \Lem~\ref{Lemma_PHI''a}.
\end{proof}

 \section{Proof of \Prop~\ref{Prop_conc}}\label{Sec_Prop_conc}

\noindent
Tools such as Azuma's inequality do not apply to the number $Z(\PHI)$ of satisfying assignments because adding or removing even a single clause could change $Z(\PHI)$ by an exponential factor.
Therefore, we prove \Prop~\ref{Prop_conc} by way of a `soft' version of the random $2$-SAT problem.
Specifically, for a real $\beta>0$ we define $Z_\beta(\PHI)$ via \eqref{eqZbeta}.
Thus, instead of dismissing assignments $\sigma\not\in S(\PHI)$ outright, we charge an $\exp(-\beta)$ penalty factor  for each violated clause.
Because the constraints are soft, showing that $\log Z_\beta(\PHI)$ concentrates is a cinch.

\begin{lemma}\label{Lemma_Azuma}
For all $t,\beta>0$ we have
$\displaystyle\pr\brk{\abs{\log Z_\beta(\PHI)-\Erw[\log Z_\beta(\PHI)]}>t\mid \vm}\leq2\exp\bc{-\frac{t^2}{2\vm\beta^2}}$.
\end{lemma}
\begin{proof}
Since adding or removing a single clause can alter $Z_\beta(\PHI)$ by at most a factor $\exp(\pm\beta)$, the assertion follows from Azuma's inequality.
\end{proof}

The following statement, whose proof relies on the interpolation method from mathematical physics, will enable us to link the random variables $\log Z_\beta(\PHI)$ and $\log Z(\PHI)$.
For a probability measure $p\in\cP(0,1)$ and $\beta >0$ let
\begin{align}\nonumber
\fB_\beta(p)&=
\Erw\brk{\log\sum_{s=\pm1}
	\prod_{i=1}^{\vd}\bc{1-\vecone\{\vs_i\neq s\}\frac{1-\exp(-\beta)}2\bc{1-\vs_i'+2\vs_i'\MU_{p,i}}}}\\
		&\qquad-\frac d2\Erw\brk{\log\bc{1-\frac{1-\exp(-\beta)}4\bc{1-\vs_1+2\vs_1\MU_{p,1}}\bc{1-\vs_2+2\vs_2\MU_{p,2}}}}.
		\label{eqBB}
\end{align}
These two expectations exist and are finite because $0\leq \beta<\infty$.
(More precisely, their absolute values are bounded by $\log 2 + \beta d$ and $\beta$, respectively.)

\begin{lemma}[{\cite[\Thm~1]{PanchenkoTalagrand}}]\label{Lemma_PT}
For any $p\in\cP(0,1)$ and any $0\leq\beta<\infty$ we have $\lim_{n\to\infty}\frac1n\Erw\log Z_\beta(\PHI)\leq \fB_\beta(p)$.
\end{lemma}

\noindent
Combining \Lem s~\ref{Lemma_Azuma} and~\ref{Lemma_PT}, we obtain the following bound for `hard' $2$-SAT.

\begin{corollary}\label{Cor_PT}
For any $\beta>0$ we have $\lim_{n\to\infty}\pr\brk{\log Z(\PHI)>n\fB_\beta(\pi_d)+n^{2/3}}=0$.
\end{corollary}
\begin{proof}
 We have $Z_\beta(\PHI)\geq Z(\PHI)$ and
\Lem s~\ref{Lemma_Azuma} and~\ref{Lemma_PT} imply  $\lim_{n\to\infty}\pr\brk{\log Z_\beta(\PHI)>n\fB_\beta(\pi_d)+n^{2/3}}=0$.
\end{proof}

\begin{proof}[Proof of \Prop~\ref{Prop_conc}]
We begin by observing that the limit $\lim_{\beta\to\infty}\fB_\beta(\pi_d)$ exists and is finite. First, there is the pointwise and monotone convergence of the integrands:
\begin{align}
\log\sum_{s=\pm1}
	\prod_{i=1}^{\vd}\bc{1-\vecone\{\vs_i\neq s\}\frac{1-\exp(-\beta)}2\bc{1-\vs_i'+2\vs_i'\MU_{\pi_d,i}}} &\stacksign{$\beta \to \infty$}{\longrightarrow} \log\sum_{s=\pm1}
	\prod_{i=1}^{\vd}\bc{1- \frac{\vecone\{\vs_i\neq s\}}{2} \bc{1-\vs_i'+2\vs_i'\MU_{\pi_d,i}}},\label{eqProp_betalim1}\\
\log\bc{1-\frac{1-\exp(-\beta)}4\bc{1-\vs_1+2\vs_1\MU_{\pi_d,1}}\bc{1-\vs_2+2\vs_2\MU_{\pi_d,2}}} &\stacksign{$\beta \to \infty$}{\longrightarrow}   \log\bc{1-\frac{1}4\bc{1-\vs_1+2\vs_1\MU_{\pi_d,1}}\bc{1-\vs_2+2\vs_2\MU_{\pi_d,2}}}.\label{eqProp_betalim2}
\end{align}
Further, since $\MU_{\pi_d}\disteq1-\MU_{\pi_d}$ by \Prop~\ref{Prop_BP} and because $1 - \vs + 2 \vs \MU_{\pi_d}$ equals either $2 \MU_{\pi_d}$ or $2 (1 - \MU_{\pi_d})$, we obtain
\begin{align}
\log\sum_{s=\pm1}
	\prod_{i=1}^{\vd}\bc{1- \frac{\vecone\{\vs_i\neq s\}}{2} \bc{1-\vs_i'+2\vs_i'\MU_{\pi_d,i}}} &\stackrel{d}{=} \log\bc{\prod_{i=1}^{\vd^-}\MU_{\pi_d,i}+\prod_{i=1}^{\vd^+}\MU_{\pi_d,i+\vd^-}}, \label{eqProp_rewritelim1}\\
 \frac d 2 \log\bc{1-\frac{1}4\bc{1-\vs_1+2\vs_1\MU_{\pi_d,1}}\bc{1-\vs_2+2\vs_2\MU_{\pi_d,2}}}  &\stackrel{d}{=} \frac d2\log\bc{1-\MU_{\pi_d,1}\MU_{\pi_d,2}}.\label{eqProp_rewritelim2}
\end{align}
Moreover, \Prop~\ref{Prop_BP} shows that the monotone limits are integrable and therefore an application of the monotone convergence theorem to (\ref{eqProp_betalim1}) and (\ref{eqProp_betalim2}), followed by the simplifications  (\ref{eqProp_rewritelim1}), (\ref{eqProp_rewritelim1}), yields the identity
\begin{align*}
\lim_{\beta\to\infty}\fB_\beta(\pi_d) = \Erw\brk{\log\bc{\prod_{i=1}^{\vd^-}\MU_{\pi_d,i}+\prod_{i=1}^{\vd^+}\MU_{\pi_d,i+\vd^-}}
-\frac d2\log\bc{1-\MU_{\pi_d,1}\MU_{\pi_d,2}}} = \fB_\infty(\pi_d) < \infty.
\end{align*}
Further, \Cor~\ref{Cor_Bethe} shows that
$\fB_\infty(\pi_d)=\lim_{n\to\infty}n^{-1}\Erw[\log(Z(\PHI)\vee1)].$
Therefore, \Cor~\ref{Cor_PT} implies that 
\begin{align}\label{eqProp_conc3}
\pr\brk{n^{-1}\log(Z(\PHI)\vee1)>\fB_\infty(\pi_d)+\eps}&=o(1)\qquad\mbox{for any $\eps>0$.}
\end{align}

To complete the proof, we upper bound
\begin{align}\label{eqProp_conc4}
n^{-1}\Erw\abs{\log(Z(\PHI) \vee 1)-\Erw[\log(Z(\PHI)\vee 1)]} \leq \Erw\abs{n^{-1}\log(Z(\PHI) \vee 1)-\fB_\infty(\pi_d)} + \abs{\fB_\infty(\pi_d) - \Erw[\log(Z(\PHI)\vee 1)]}.
\end{align} 
Due to \Cor~\ref{Cor_Bethe}, the second term on the r.h.s.\ of \eqref{eqProp_conc4} tends to zero.
On the other hand, \eqref{eqProp_conc3} and \Cor~\ref{Cor_Bethe} yield that for any $\eps >0$,
\begin{align*}
 \Erw\abs{n^{-1}\log(Z(\PHI) \vee 1)-\fB_\infty(\pi_d)} \leq  \Erw\brk{\fB_\infty(\pi_d) - n^{-1}\log(Z(\PHI) \vee 1)} + 2\eps + o(1) = 2\eps + o(1),
\end{align*}
 as desired.
\end{proof}

\subsection*{Acknowledgment}
We thank Andreas Galanis and Leslie Goldberg for helpful discussions.

\end{document}